\documentclass[final,10pt]{siamltex}
\usepackage[latin1]{inputenc}
\usepackage{amsmath,amstext,amsbsy,amsopn,amscd,amsxtra,upref,amssymb}
\usepackage{amsfonts}
\usepackage{graphicx}
\usepackage{algorithmic}
\usepackage{algorithm}
\usepackage[]{mcode}
\usepackage{xspace}
\newtheorem{remark}[theorem]{Remark}
\newtheorem{example}[theorem]{Example}
\newcommand{\Id}{\mathbb I}
\newcommand{\SO}{\mathbb S}
\newcommand{\Xs}{\mathbb X}

\newcommand{\M}{\mathbb M}
\newcommand{\U}{\mathsf U}
\newcommand{\G}{\mathsf G}

\newcommand{\W}{\mathsf W}

\newcommand{\R}{\mathsf R}
\newcommand{\T}{\mathsf T}
\newcommand{\K}{\mathsf K}
\newcommand{\Z}{\mathsf Z}
\newcommand{\f}{\mathbf f}
\newcommand{\C}{\mathbf c}
\newcommand{\Q}{\Xi}
\newcommand{\Lw}{\mathsf L}

\newcommand{\deim}{\textsf{DEIM}\xspace}
\newcommand{\dime}{\textsf{Q-DEIM}\xspace}
\newcommand{\dimer}{\textsf{Q-DEIMr}\xspace}
\DeclareMathOperator*{\argmax}{arg\,max}

		\author{Zlatko Drma\v{c}\thanks{Faculty of Science, Department of Mathematics, University of Zagreb,
		Bijeni\v{c}ka 30, 10000 Zagreb, Croatia.} \and Serkan Gugercin\thanks{Department of Mathematics,
		Virginia Polytechnic Institute and State University,
		460 McBryde, Virginia Tech,
		Blacksburg, VA 24061-0123.}}
\title{A New Selection Operator for the Discrete Empirical Interpolation Method -- improved a priori error bound and extensions}
\begin{document}
	\maketitle
	\begin{abstract}
		This paper introduces a new framework for constructing the Discrete Empirical Interpolation Method (\deim) projection operator. \textcolor{black}{The interpolation node selection procedure is formulated using the QR factorization with column pivoting, and it enjoys a sharper error bound for the \deim projection error. }
		\textcolor{black}{Furthermore, for a subspace $\mathcal{U}$ given as the range of an orthonormal $\U$,  
		the \deim projection does not change if $\U$ is replaced by $\U \Omega$ with arbitrary unitary matrix $\Omega$.} 		
		\textcolor{black}{In a large-scale setting, the new approach allows modifications that use only randomly sampled rows of $\U$, but with the potential of producing good approximations with corresponding probabilistic error bounds.}
		\textcolor{black}{Another salient feature of the new framework is that robust and efficient   software implementation is easily developed, based on readily available high performance linear algebra packages.}
	\end{abstract}



\section{Introduction}\label{S=Intro}
Direct numerical simulation of dynamical systems plays a crucial role in studying 
a great variety of complex physical phenomena in areas 
ranging from neuron modeling to microchip design. 
The ever-increasing demand for accuracy 
leads to dynamical systems of ever-larger scale and complexity.
Simulation in such large-scale settings can make overwhelming 
demands on computational resources; thus creating a need for model reduction to 
create  smaller, faster approximations to complex dynamical systems that still guarantee high fidelity. 

\subsection{Model Reduction by Galerkin Projection}
Consider the following nonlinear dynamical system of ordinary differential equations (ODE)
\begin{equation} \label{nlfom}
E\, \dot{x}(t)  = A \,x(t)  + \f(x(t)) + B\,\mathbf{g}(t) ,\;\; \textcolor{black}{t\geq 0,}
\end{equation}
where $E, A\in \mathbb{R}^{n\times n}$, $B\in \mathbb{R}^{n\times \nu}$, $\f: \mathbb{R}^{n} \to\mathbb{R}^{n} $ and $\mathbf{g}: \textcolor{black}{[0,\infty)} \to\mathbb{R}^{\nu} $. In (\ref{nlfom}), $x(t) \in \mathbb{R}^n$ is  the state and $\mathbf{g}(t)$ is the external forcing term (input); thus (\ref{nlfom}) has  $n$ degrees of freedom and $\nu$ inputs.

Systems of the form (\ref{nlfom}) with very large state-space dimension ($n \approx O(10^6)$ or higher) arise in many disciplines and typically originate from discretization of partial differential equation models.  The goal of model reduction is to replace (\ref{nlfom}) with a reduced surrogate dynamical system  having much lower state space dimension, $r\ll n$.  The reduced model \textcolor{black}{will then} have the structure 
\begin{equation} \label{nlrom}
E_r\, \dot{x}_r(t)  = A_r \,x_r(t)  + \f_r(x_r(t)) + B_r\,\mathbf{g}(t),
\end{equation}
where $E_r, A_r\in \mathbb{R}^{r\times r}$, $B_r\in \mathbb{R}^{r\times \nu}$, and $\f_r: \mathbb{R}^{r} \to\mathbb{R}^{r} $.

We will use a Galerkin projection to construct  
 the reduced model (\ref{nlrom}): 
Let  ${\mathcal V}_r$ be an $r$-dimensional subspace \textcolor{black}{spanned by the columns of} $V \in \mathbb{R}^{n\times r}$. Then, we approximate the full-state
$x(t)$ using the ansatz $x(t) \approx V x_r(t)$ and enforce the Galerkin orthogonality condition 
\textcolor{black}{$\displaystyle
E V\dot{x}_r(t) - A V x_r(t) -\f\big(V x_r(t)\big)-B \mathbf{g}(t)\perp {\mathcal V}_r
$}
to obtain the reduced model  (\ref{nlrom}) with the reduced model quantities given by
\begin{equation}  \label{red_projection}
E_r= V^{T} E V,~~A_r= V^{T} A V,~~B_r = V^TB,~{\rm and}~~  \f_r(x_r(t))=V^T\f(Vx_r(t)).
\end{equation}
\subsection{Galerkin Projection using Proper Orthogonal Decomposition}
For linear dynamical systems, i.e., when $\f = \mathbf{0}$ in (\ref{nlfom}), a plethora of methods exist to perform  model reduction: These include gramian based methods such as Balanced Truncation \cite{mullis1976synthesis,moore1981principal} and Optimal Hankel Norm Approximation  \cite{glover1984all} or rational interpolation based methods such Iterative Rational Krylov Algorithm \cite{gugercin2008hmr}. 
 These methods rely on the concept of transfer function and perform   model reduction independent of the input $\mathbf{g}(t)$. These ideas have been recently extended to systems with bilinear \cite{bai2006projection,BreitenBenner2012b,flagg2013multipoint} and quadratic nonlinearities \cite{gu2011qlmor,benner2015two}. \textcolor{black}{For} general nonlinearities,  Proper Orthogonal Decomposition (POD) is the most-commonly used method. POD \cite{lumley,berkooz}
obtains the model reduction basis  $V$ from a  truncated SVD
approximation to a matrix of ``snapshots'', a numerically computed trajectory of the full model. 
It is related to methods (and known \textcolor{black}{by} other names) \textcolor{black}{such as}
Principal Component Analysis (PCA) in statistical analysis {\cite{hotelling}} and  Karhunen-Lo\'{e}ve expansion  \cite{loeve}  in stochastic analysis.

To construct the model reduction basis $V$ via POD, one performs a numerical simulation of (\ref{nlfom}) for an input $\mathbf{g}(t)$ and initial condition $x_0$. Let $  x_0, x_1, \ldots,   x_{N-1}$ denote the {\it snapshots} resulting from this numerical simulation; i.e, $x_i = x(t_i) \in \mathbb{R}^n$ for $i=0,1,\ldots,N-1$.  Construct the {\it POD snapshot matrix}
\begin{equation} \label{Xsnap}
\Xs = \left[~x_0, x_1, x_2, \ldots,x_{N-1}\right] \in \mathbb{R}^{n \times N}
\end{equation}
  and compute its thin SVD
\begin{equation}
  \Xs =   Z \boldsymbol{\Sigma}   Y^T,
\end{equation}
where $  Z \in \mathbb{R}^{n \times N}$, $\boldsymbol{\Sigma} \in \mathbb{R}^{N \times N}$,  and $Y \in \mathbb{R}^{N \times N}$ with $Z^TZ = Y^T Y = \Id_N$, and $\boldsymbol{\Sigma} = \mbox{diag}(\sigma_1, \sigma_2, \ldots, \sigma_{n_s})$, with $\sigma_1 \geq \sigma_2 \geq \ldots \geq \sigma_{N} \geq 0$.
Then model reduction by POD chooses  $ V$  as the leading $r$ left singular vectors of $\Xs$  corresponding to the $r$ largest singular values. Using  MATLAB notation, this corresponds to $V = Z(:,1:r)$. This basis selection by POD minimizes  $\sum_{i=0}^{N} \|   x_i -   \Phi   \Phi^T   x_i \|_2^2$ over all \textcolor{black}{$\Phi \in \mathbb{R}^{n\times r}$} with orthonormal columns. 
\textcolor{black}{Since the objective function does not change if $\Phi$ is post-multiplied by an arbitrary $r\times r$ orthogonal matrix, this procedure actually seeks an $r$--dimensional subspace that optimally captures the snapshots in the least squares sense.} 
For more details on POD, we refer the reader to 
\cite{hinze2005proper,kunisch2002galerkin}.

\subsection{The lifting bottleneck}
Even though the state $x_r(t)$ of the reduced model  (\ref{nlrom}) lives in an $r$-dimensional subspace, 
definition of the reduced nonlinear term $\f_r(x_r(t)) = V^T\f(Vx_r(t))$  in (\ref{red_projection}) requires
lifting $x_r(t)$ back to the full $n$-dimensional subspace in order to evaluate the nonlinear term; this is known as the lifting bottleneck and degrades the performance of reduced models for nonlinear systems.
Various approaches exist to tackle this issue; see, e.g.,  \cite{Everson1995,barrault04-EIM,Astrid2008, DEIM,Carlberg2013}. In this paper, we focus on
 the \emph{Discrete Empirical Interpolation Method} (\deim) \cite{DEIM}, a discrete variant of the Empirical Interpolation Method introduced in \cite{barrault04-EIM}. 

As explained in the original source \cite{DEIM}, \deim can be used to approximate and efficiently evaluate a general nonlinear function $\f$, which is not necessarily tied to the  model reduction set-up we discussed above. For example, $\f(\tau)$ could  be a vector-valued function of possibly multidimensional parameter $\tau$. Therefore, following \cite{DEIM}, we will present the \deim construction and our analysis for a generic nonlinear vector valued function   $\f(\tau)$, yet will point out the implications for nonlinear model reduction.

\subsection{\deim}
Given a nonlinear function $\f:\mathcal{T}\longrightarrow \mathbb{R}^n$ 
\textcolor{black}{with $\mathcal{T} \subset  \mathbb{R}^d$}
 and a matrix  $\mathsf{U}\in\mathbb{R}^{n\times m}$ of rank $m$, \deim  approximation of $\f$ is defined by \cite[Definition 3.1]{DEIM}
\begin{equation} \label{eq_deim}
\widehat{\f}(\tau) = \U (\SO^T \U)^{-1}\SO^T \f(\tau) ,
\end{equation}
where $\SO$ is $n\times m$ matrix obtained by selecting certain columns of the $n\times n$ identity matrix $\Id$.   
With the \deim approximation to $\f$ defined as in (\ref{eq_deim}), the  nonlinear term in the reduced model
(\ref{nlrom}) is now \textcolor{black}{approximated} by
\begin{equation} \label{eq_deim_fr}
\f_r(x_r(t)) \textcolor{black}{\approx} V^T \U (\SO^T \U)^{-1}\SO^T \f(Vx_r(t)).
\end{equation}
An effective numerical implementation of $\f_r(x_r(t))$ is different than its analytical formula in (\ref{eq_deim_fr})  and allows computing $\f_r(x_r(t))$ without lifting $x_r(t)$ to the full dimension $n$ and by only selecting a certain rows of $V x_r(t)$.  We skip those details and refer the reader to \textcolor{black}{\cite[\S 3.4]{DEIM}}.
 
 \paragraph{Computation of the \deim basis $\U$}

In an application, the matrix $\U$ can be computed as follows. For a finite grid $\mathcal{T}_{\boxplus}\subset\mathcal{T}$, 
the function is sampled at $\tau_j\in\mathcal{T}_{\boxplus}$ and, as done for state $x(t)$ in POD for model reduction,  the function values, {\it nonlinear snapshots},  are collected in a matrix $F$, i.e.,
$F = [\f(\tau_1),\f(\tau_2),\ldots,\f(\tau_\kappa)]$.
If $\f(\tau)$ is $n_1\times n_2$ matrix valued, the $\mathrm{vec}(\cdot)$ operator is used to map its range to $\mathbb{{R}}^{n_1\cdot n_2}$.  Then, an orthogonal projection $\Omega=\U\U^T$, of
low rank $m$, onto the range  $\mathcal{U}=\mathcal{R}(\U)$
is constructed so that $\| F - \Omega F\|_F$ is minimal. Typically, $m\ll n$.
Therefore, $\U$ can be considered as the POD basis for the nonlinear snapshots.
 The hope is that the range of $\U$ will capture the values of $\f$ over the entire parameter space, i.e.,
$\|\f(\tau) - \U\U^T \f(\tau)\|_2$ will be sufficiently small at any $\tau\in\mathcal{T}$.

The role of $\SO$, which we will call \emph{selection operator}, is to strategically pick coordinate indices in $\mathbb{R}^n$ at which the approximant interpolates
$\f$. (Note that $\SO^T \widehat{\f}(\tau) = \SO^T \f(\tau)$.)
The \deim algorithm, proposed in \cite{DEIM}, forces the selection operator $\SO$ to seek $m$ linearly independent rows of $\U$ such that the local growth of the spectral norm of $(\SO^T \U)^{-1}$ is limited via a greedy search, as implemented in Algorithm \ref{ALG:DEIM}.
This objective is founded in the following theoretical basis of \deim \cite[Lemma 3.2]{DEIM}:	
\begin{lemma}\label{Lemma:deim0}
		Let $\U\in\mathbb{R}^{n\times m}$ be orthonormal ($\U^* \U=\Id_m$, $m<n$) and let 
		\begin{equation}\label{eq:D*f}
		\widehat{f} = \U(\SO^T \U)^{-1}\SO^T f 
		\end{equation}
		be the \emph{\deim} projection of an arbitrary $f\in\mathbb{R}^n$, with $\SO$ computed by Algorithm \ref{ALG:DEIM}. Then 
	\begin{equation}\label{eq:DEIM-error0}
	\| f - \widehat{f} \|_2 \leq \C \| (\Id-\U\U^*)f\|_2,\;\;\C = \| (\SO^T \U)^{-1} \|_2 ,
	\end{equation}
	where 
	$$
	\C \leq \frac{(1+\sqrt{2n})^{m-1}}{\|u_1\|_{\infty}} \leq \sqrt{n} (1+\sqrt{2n})^{m-1} .
	$$	
\end{lemma}

\begin{algorithm}[hh]
	\caption{\textcolor{black}{\deim (Discrete Empirical Interpolation Method)} \cite[Algorithm 1]{DEIM}} \label{ALG:DEIM}
	\begin{algorithmic}[1]
		\STATE \textbf{Input:} $u_1,\ldots, u_m$ linearly independent.
		\STATE \textbf{Output:} Selection operator $\SO = \SO_m$ (implicitly by $\wp_m$). 
		\STATE $ p_1 = \argmax_i(|u_1(i)|)$ ; $\U_1=[u_1]$; $\SO_1=[e_{p_1}]$ ; $\wp_1= [p_1]$
		\FOR{$j=2 : m$}
		\STATE Solve $\SO_{j-1}^T\U_{j-1} z= \SO_{j-1}^T u_j$ for $z$ ; 
		\STATE $r_j = u_j - \U_{j-1} z$ ;
		 $p_j = \argmax_i(|r_j(i)|)$ ; 
		\STATE $\U_j = [ \U_{j-1}, u_j]$ ; $\SO_j=[\SO_{j-1} , e_{p_j}]$ ; $\wp_j = (\wp_{j-1} , p_j)$ ;
		\ENDFOR
	\end{algorithmic}
\end{algorithm}	
%
%
Hence, we can focus on a pure matrix theoretical problem:\footnote{From now on, we consider the problem over the complex field.} \emph{Given orthonormal $\U\in\mathbb{{C}}^{n\times m}$ ($\U^*\U=\Id_m$) find a row selection matrix $\SO$ with $\|(\SO^T\U)^{-1}\|_2$ as small as possible.} If $\mathcal{R}(\mathsf{U})$ captures the behavior of $\f$ well over the given parameter space, and if $\SO$ results in a moderate
value of $\C$ in  (\ref{eq:DEIM-error0}), the \deim approximation will succeed. 

The error bound (\ref{eq:DEIM-error0}) in Lemma \ref{Lemma:deim0} is rather pessimistic and the \deim projection usually performs substantially better in practice, 
see \cite{DEIM} for several illustrations of superior performance of \deim. Hence, \textcolor{black}{this} is an interesting  theoretical question:  \emph{can the upper bound can be improved, and what selection operator $\SO$ will have \textcolor{blue}{a} sharper a priori error bound, perhaps only mildly dependent on $n$?}.


Note that $\SO$ computed in Algorithm \ref{ALG:DEIM}  depends on a particular basis for $\mathcal{U}$; just reordering the basis vectors may result in different $\SO$. If, for example, $\U$ consists of the left singular vectors of the $m$ dominant singular values of the data samples matrix $F$, and if some of those singular values are multiple or tightly clustered, then some singular vectors (columns of $\U$) are non-unique or are numerically \textcolor{black}{badly determined by the data} and the computed $\U$ could be algorithm dependent. But the subspace they span is well-determined. Therefore, from both \textcolor{black}{the} theoretical and  practical points of view, it is \textcolor{black}{important to ask} \emph{whether we can efficiently construct 
$\SO$ with an a priori assurance that $\C$ will be moderate and independent of the choice of an orthonormal basis $\U$ of $\mathcal{U}$.}

Our interest for studying \deim in more detail was triggered by the above theoretical questions from a numerical linear algebra point of view, and by a practical question of efficient implementation of \deim as  mathematical software on high performance computing machinery.
\textcolor{black}{In \S \ref{SS:DEIM-LUPP}, the complexity of Algorithm \ref{ALG:DEIM} is estimated to be $O(m^2 n)+O(m^3)$. Unfortunately,} it has unfavorable flop per memory reference ratio (level 2 BLAS) which precludes efficient software implementation. It would be advantageous to have an algorithm based on BLAS 3 building blocks, with potential for parallel implementations. Furthermore, we may ask \emph{whether the contribution of the factor $n$ in the overall complexity can be reduced or even removed (e.g. using only a subset of the rows of $\mathsf{U}$) without substantial loss in the quality of the computed selection operator.}  

Fortunately,  \textcolor{black}{an} affirmative answer to all  \textcolor{black}{the} questions above is surprisingly simple and effective: QR factorization with column pivoting of $\U^*$. Our new implementation of \deim,
designated as \dime ,
computes $\SO$ independent of a particular orthonormal basis $\U$,  enjoys a better upper bound for the condition number $\C$ of the \deim projection, and in practice computes $\SO$ with usually smaller value of $\|(\SO^T\U)^{-1}\|_2$ than the original \deim algorithm.
A further advantage of \dime is that it is based on numerically robust high performance procedures, already available in software packages \textcolor{black}{such} as LAPACK, ScaLAPACK, MATLAB, so no additional effort is needed for 
\textcolor{black}{tuning} high performance \deim. The details and a theoretical foundation of \dime are given in \S \ref{S=dime}. In particular, in \S \ref{SS=dime} we provide a selection procedure and theoretical analysis showing that the \deim projection is almost as good as the orthogonal projection onto the range of $\U$. Numerical experiments that illustrate the performance of \dime in the context of nonlinear model reduction  are presented in \S \ref{S=examples}. In \S \ref{SS=dimer}, we show that accurate \deim projection is possible even with using only a small portion of the rows of $\U$, and we introduce \dimer, a restricted and randomized \deim selection that combines the technique used in \dime with the ideas of 
randomized sampling. Further developments and applications 
are outlined in \S \ref{S=Conclusion}.

\section{A new \deim framework}\label{S=dime}  
A key observation leading to a selection strategy presented in this section is based on a solution to a similar problem in \cite{drmac-block-jacobi}, arising in the proof of global convergence of a block version of the Jacobi algorithm for diagonalization of Hermitian matrices. There, a row permutation is needed such that the $(1,1)$ diagonal block of a family of $2\times 2$ block partitioned unitary matrices has a uniform lower bound for its smallest singular value, independent of the family and only depending on the parameters of the partition (block dimensions). 

It is clear that the
selection of well conditioned submatrices is deeply connected with rank \textcolor{black}{revelation}, and that in fact the most reliable rank revealing QR factorizations are indeed based on selecting certain well conditioned submatrices, see e.g. \cite{chandras-ipsen-rrqr-94}. Note, however, that in our case here, the rank is not an issue, as our matrix $\U$ is orthonormal. 

We adapt the strategy from \cite{drmac-block-jacobi} and use it in \S \ref{SS=dime} as a basis for introducing a new framework for construction of the \deim projection; the result is a new selection method, called \dime, with an improved theoretical bound on $\C$ and with the selection operator invariant under arbitrary changes of the orthonormal basis of the range of $\U$. We also use the seminal work of Goreinov, Tyrtyshnikov and Zamarshkin \cite{Goreinov19971} to show that \deim projection is not only numerically but also theoretically almost as good as the orthogonal projection, up to a factor of the dimension. 

\subsection{\dime -- a new selection procedure}\label{SS=dime}
An answer to all practical questions raised in \S \ref{S=Intro} is given in the following theorem.
Its constructive proof is based on \cite{drmac-block-jacobi}, but we provide all the details for the reader's convenience, and also because we need them in the further developments in  \S \ref{SS=dimer}.
	
\begin{theorem}\label{Theorem:2mbound}
Let $\mathsf{U}\in\mathbb{C}^{n\times m}$, $\mathsf{U}^*\mathsf{U}=\Id_m$, $m<n$. Then :
\begin{itemize}
\item  There exists an algorithm to compute a selection operator $\SO$ with complexity $O(nm^2)$, such that
\begin{equation}\label{eq:2m-bound}
\| (\SO^T \mathsf{U})^{-1}\|_2  \leq {\sqrt{n-m+1}}\,\, \frac{\sqrt{4^m+6 m - 1}}{3},
\end{equation}	
and for any $f\in\mathbb{C}^n$
\begin{equation}\label{eq:dime-error1}
\| f - 	\U(\SO^T \U)^{-1}\SO^T f \|_2 \leq \sqrt{n} \,O(2^m)\, \| f - \U\U^* f \|_2 .
\end{equation}
If $\mathsf{U}$ is only full column rank, then the bound (\ref{eq:2m-bound}) changes to
\begin{equation}\label{eq:2m-bound1}
\| (\SO^T \mathsf{U})^{-1}\|_2 \leq \frac{\sqrt{n-m+1}}{\sigma_{\min}(\mathsf{U})} \,\frac{\sqrt{4^m+6 m - 1}}{3}.
\end{equation}

\item There exists a selection operator $\SO_{\star}$ such that the \deim projection error is bounded by 
\begin{equation}\label{eq:dime-error2}
\| f - 	\U(\SO_{\star}^T \U)^{-1}\SO_{\star}^T f \|_2 \leq \sqrt{1+m(n-m)} \, \| f - \U\U^* f \|_2 .
\end{equation}
\item The selection operators $\SO$, $\SO_{\star}$ do not change if $\mathsf{U}$ is changed to 
$\mathsf{U} \Omega$, where $\Omega$ is arbitrary $m\times m$ unitary matrix, i.e., the selection of indices is assigned to a point on the Stiefel manifold, represented by $\U$.
\end{itemize}
\end{theorem}
\begin{proof}
	Let $\W=\U^* \in \mathbb{C}^{m\times n}$, and let 
	\begin{equation}\label{eq:WP=QR}
	\W\Pi = \begin{pmatrix} \widehat{\W}_1 & \widehat{\W}_2
	\end{pmatrix} = \mathsf{Q} \R
	= {\mathsf{Q}} \left(\begin{array}{cccc|ccc} 
	* & * & * & * & * & * & * \cr
	0 & * & * & * & * & * & * \cr
	0 & 0 & * & * & * & * & *  \cr
	0 & 0 & 0 & * & * & * & *
	\end{array}\right)
	\end{equation}
	be a column pivoted (rank revealing) QR factorization. We have at our disposal 
	a variety of pivoting strategies that reveal the numerical rank by constructing
	$\R$ in a way to control the condition numbers (explicitly or implicitly) of its 
	leading submatrices.
	
	For instance, the Businger--Golub pivoting \cite{bus-gol-65} at step $i$  first determines a smallest local index $\hat p_i$ of the
	largest (in Euclidean norm) column in the submatrix $(i:m,i:n)$ and swaps globally the columns $i$ and $p_i=i-1+\hat p_i$ in the whole matrix.
	The following scheme illustrates the case with $n=7$, $m=4$, $i=2$, $\hat p_2=3$, and $p_2=4$:
	\textcolor{black}{
	\begin{equation}\label{eq:pivoting-sheme}
	\bordermatrix{
		&       &   i   &       &    p_i     &   & & n  \cr
		& \star & \star & \star & \star   & \star  & \star & \star\cr
	  i &   0   &  \bullet    &  *    & \circledast   & * & * & *\cr
		&   0   &   \bullet   &   *   & \circledast   & * & * & *\cr
      m &   0   &  \bullet    &   *   & \circledast   & * & * & *
	} \stackrel{swap(i,p_i)}{\rightarrow\!\longrightarrow\!\longrightarrow}
	\bordermatrix{
		&   & i &   & p_i &   & & n  \cr
		& \star & \star & \star & \star & \star  & \star & \star\cr
	i   & 0 & \circledast & * & \bullet   & * & * & *\cr
		& 0 & \circledast & * & \bullet   & * & * & *\cr
	m   & 0 & \circledast & * & \bullet   & * & * & *
	} .
	\end{equation}
}
	Then, the QR step maps the $i$--th column \textcolor{black}{in the sub-matrix $(i:m,i:n)$} to $\mathsf{e}_i\mathsf{R}_{ii}$ and keeps all
	the remaining column norms in the submatrix unchanged and bounded by $|\mathsf{R}_{ii}|$.
	(Here $\mathsf{e}_i$ denotes  $i$--th canonical vector of appropriate dimension.)
	The product of all transpositions gives the
	permutation $\Pi$.

	We define the selection operator $\SO$ as the one that collects the columns of $\W$ to build $\widehat{\W}_1$; this implies that $\SO^T\U = \widehat{\W}_1^*$ and we need to
	estimate $\|\widehat{\W}_1^{-1}\|_2$. 
	Partition $\mathsf{R}$ in (\ref{eq:WP=QR}) as $\mathsf{R}=\begin{pmatrix} \mathsf{T} & \mathsf{K}\end{pmatrix}$
	with $m\times m$ upper triangular $\mathsf{T}$. Then $\widehat{W}_{1} = \mathsf{Q}\mathsf{T}$, and the
	problem reduces to bounding $\|\mathsf{T}^{-1}\|_2$.
	As a result of the pivoting (\ref{eq:pivoting-sheme}), the matrix $\mathsf{T}$, as the
	leading $m\times m$ submatrix of $\mathsf{R}$, has a special diagonal dominance structure:
	\begin{equation}
	|\mathsf{T}_{ii}|^2 \geq \sum_{j=i}^k |\mathsf{T}_{jk}|^2,\;\; 1\leq i \leq k \leq m;\;\;\;
	|\mathsf{T}_{mm}| = \max_{j=m:n} |\mathsf{R}_{mj}|.\label{eq:T_ii-dominance}
	\end{equation}
	Further, since $\widehat{\W}\equiv \W\Pi=\mathsf{Q}\mathsf{R}$ and since $\widehat{\W}\widehat{\W}^* = \U^*\U = 
	\mathsf{Q} \R \R^* \mathsf{Q}^* = {\Id}_m$, we conclude that $\R\R^*=\Id_m$, which
	implies that
	\begin{equation}
	1 = \|\mathsf{R}(m,:)\|_2 = |\mathsf{T}_{mm}|^2 + \sum_{j=m+1}^n |\mathsf{R}_{mj}|^2 \leq (n-m + 1) |\mathsf{T}_{mm}|^2 ,
	\end{equation}
and that
	\begin{equation}\label{eq:Tmm}
	\min_{i=1:m} |\mathsf{T}_{ii}| = |\mathsf{T}_{mm}| \geq \frac{1}{\sqrt{n-m + 1}}.
	\end{equation}	
	If we set $\mathsf{D}=\mathrm{diag}(\mathsf{T}_{ii})_{i=1}^m$, $\breve{\mathsf{T}} = \mathsf{D}^{-1}\mathsf{T}$, then $\|\mathsf{T}^{-1}\|_2 \leq \sqrt{n-m+1}\|\breve{\mathsf{T}}^{-1}\|_2$.
		Further, if we assume $\mathsf{U}$ to be just of rank $m$, not necessarily orthonormal, then
		\begin{eqnarray}
		|\mathsf{T}_{mm}| &\geq& \frac{\|\mathsf{R}(m,:)\|_2}{\sqrt{n-m+1}}\geq\frac{\sigma_{\min}(\mathsf{R})}{\sqrt{n-m+1}} = 
		\frac{\sigma_{\min}(\mathsf{U})}{\sqrt{n-m+1}}, \label{eq:Tmm1}\\
		\sigma_{\min}(\mathsf{T}) &\geq& \frac{\sigma_{\min}(\mathsf{U})}{\sqrt{n-m+1}} \frac{1}{\|\breve{\mathsf{T}}^{-1}\|_2} .\label{eq:Tmm2}
		\end{eqnarray}
\textcolor{black}{Since $\SO^T\U=\widehat{\W}_1^*=\T^* \mathsf{Q}^*$, it follows that $\|(\SO^T \U)^{-1}\|_2=\|\T^{-1}\|_2=1/\sigma_{\min}(\T)$.	
	Hence, to prove (\ref{eq:2m-bound}) and (\ref{eq:2m-bound1}) it remains to estimate the norm of $\breve{\mathsf{T}}^{-1} = \mathsf{T}^{-1} \mathsf{D}$.} This can be done using an analysis of 	
	Faddeev, Kublanovskaya and Faddeeva \cite{FKF-1968}, that can also be found in 
	\cite{law-han-74}.  Systematic use of (\ref{eq:T_ii-dominance}) \textcolor{black}{as in \cite[Chapter 6]{law-han-74}}
	yields the following useful inequalities
	\begin{displaymath}
	|\mathsf{T}^{-1} \mathsf{e}_i| \leq \frac{1}{|\mathsf{T}_{ii}|}
	\begin{pmatrix} 2^{i-2}\!\!, & 2^{i-3}\!\!, & \ldots ,& 4, & 2, & 1, & 1, & 0,  & \ldots & 0 \end{pmatrix}^{T},
	\;\;i=2,\ldots, m,
	\end{displaymath}
	where the absolute value and the inequality between vectors
	are understood element--wise. For $i=1$, trivially, we have $\mathsf{T}^{-1}\mathsf{e}_1 = \mathsf{e}_1({1}/{\mathsf{T}_{11}})$,
	and $\breve{\mathsf{T}}^{-1}\mathsf{e}_1 = \mathsf{e}_1$.
	For $i=2,\ldots, m$ we use the relations
	$
	\breve{\mathsf{T}}^{-1} \mathsf{e}_i = \mathsf{T}^{-1}\mathsf{D}\mathsf{e}_i =
	\mathsf{T}^{-1}\mathsf{e}_i {\mathsf{T}_{ii}}
	$
	to conclude
	\begin{displaymath}
	{\displaystyle |\breve{\mathsf{T}}^{-1} \mathsf{e}_i| \leq
		\begin{pmatrix} 2^{i-2}\!\!, & 2^{i-3}\!\!, & \ldots ,& 4, & 2, & 1, & 1, & 0,  & \ldots & 0 \end{pmatrix}^{T}},
	\end{displaymath}
	and thus (\ref{eq:2m-bound}), (\ref{eq:2m-bound1}) follow by (\ref{eq:Tmm}), (\ref{eq:Tmm1}), (\ref{eq:Tmm2}) and	
	$$
	\|\breve{\mathsf{T}}^{-1}\|_2\leq \|\breve{\mathsf{T}}^{-1}\|_F \leq g(m),\;\;
	\mbox{where}\;\; g(m) = \sqrt{m + \sum_{i=2}^m \sum_{j=0}^{i-2} 4^j}=\frac{\sqrt{4^m+6 m - 1}}{3}.
	$$
	
\noindent	If $\mathsf{U}$ is changed to $\mathsf{U}\Omega$ with unitary $\Omega$, then the column pivoted QR is computed with $\widetilde{\mathsf{W}}=\Omega^*\mathsf{U}^*=\Omega^*\mathsf{W}$ on input. The fact that the QR factorization of $\mathsf{W}$ or of $\widetilde{\mathsf{W}}$ is implicitly the Cholesky factorization of the Hermitian semidefinite $\mathsf{H}=\mathsf{W}^*\mathsf{W}=\widetilde{\mathsf{W}}^* \widetilde{\mathsf{W}}$ extends to the pivoted factorizations as well. 
In the first step, obviously, looking for the largest diagonal entry of $\mathsf{H}$
in the pivoted Cholesky factorization is equivalent to looking for the column of $\widetilde{\mathsf{W}}$ (or $\mathsf{W}$) of largest Euclidean length. 
Hence, the pivoting will select the same columns in both cases. After $k$ steps
of annihilations using Householder reflectors with appropriate column interchanges, the
intermediate result is
$
\widetilde{\mathsf{W}}^{(k)} = 
\left(\begin{smallmatrix}
\widetilde{\mathsf{W}}^{(k)}_{[11]} &  \widetilde{\mathsf{W}}^{(k)}_{[12]} \cr
0 & \widetilde{\mathsf{W}}^{(k)}_{[22]}
\end{smallmatrix}\right) ,
$
and it is easily checked that $(\widetilde{\mathsf{W}}^{(k)}_{[22]})^*\widetilde{\mathsf{W}}^{(k)}_{[22]}$ equals the Schur complement at the corresponding step in the pivoted Cholesky factorization of $\mathsf{H}$. Hence, the next step will have the same pivot selection in both processes. 

The existence of $\SO_{\star}$ is based on an elegant argument by Goreinov et al. \cite{Goreinov19971}, who used the concept of matrix volume (the absolute value of the determinant).
The selection $\SO_{\star}$ is defined to be the one that maximizes the volume of
$\SO_{\star}^T \U$ over all ${n\choose m}=\frac{n!}{m!(n-m)!}$ $m\times m$ submatrices of $\U$.
Then, by \cite[Lemma 2.1]{Goreinov19971}, 
\begin{equation}\label{eq:S*}
\|(\SO_{\star}^T\U)^{-1}\|_2\leq\sqrt{1+m(n-m)} .
\end{equation}
Since postmultiplying $\U$ by a unitary $\Omega$ cannot change the volume of any $m\times m$ submatrix of $\U$, the same maximizing volume submatrix will be selected. (Here we assume that \textcolor{black}{a} simple additional condition is imposed to assure unique selection in \textcolor{black}{the} case of several maxima. For instance, among multiple choices, select the one with smallest lexicographically ordered indices.)
The error bounds (\ref{eq:dime-error1}) and  (\ref{eq:dime-error2}) follow by inserting the corresponding bounds (\ref{eq:2m-bound}) and (\ref{eq:S*}) for $\C$ into (\ref{eq:DEIM-error0}).
\end{proof}
	
\begin{remark}
{\em
\textcolor{black}{
According to \cite{FKF-1968}, a slightly better bound $\sqrt{(4^{m-1}+2)/3}$ can be used instead of $g(m)$. It should be emphasized that the $O(2^m)$ upper bound is attained only on a contrived example (the notorious Kahan matrix \cite{kahan-66}) and in practice it can be replaced by $O(m)$.   Also, in the application of \deim, $m$ is assumed small to modest, so $\|\breve{\mathsf{T}}^{-1}\|_2$	can be estimated in $O(m^2)$ or even $O(m)$ time using a suitable condition number estimator; the factor $\sqrt{n-m+1}$ can be replaced by \textcolor{black}{the} actually computed and potentially smaller value $1/|\mathsf{T}_{mm}|$. The deployment of an incremental condition estimator will be particularly important in a randomized sampling version of \dime introduced in \S \ref{SS=dimer}. 
}
}
\end{remark}

More sophisticated rank revealing QR
factorization can further reduce the upper bound on $\C$, but practical experience shows that the pivoting used in Theorem \ref{Theorem:2mbound} works very well. It has been conjectured in \cite{Goreinov19971} that (\ref{eq:S*}) can be replaced with $\|(\SO_{\star}^T\U)^{-1}\|_2\leq\sqrt{n}$, and proved that no bound
smaller than $\sqrt{n}$ can exist in general. 

\begin{remark}  \label{rem:volume}
{\em
While the existence and superiority of $\SO_{\star}$ are clear, its construction is difficult. However, we can use its characterization to understand why the selection operator $\SO$ defined in 
Theorem \ref{Theorem:2mbound} (and the Businger-Golub pivoting in general) usually works very well in practice. The volume of the submatrix selected by $\SO$ equals the volume $\prod_{i=1}^m|\T_{ii}|$ of the upper triangular $\T$, which is the leading $m\times m$ submatrix of the computed $\R$ factor.
On the other hand, the pivoting, by design, at each step tries to produce maximal possible $|\T_{ii}|$; thus it can be interpreted as a greedy volume maximizing scheme. In fact, such an interpretation motivates post-processing to
increase the determinant, e.g. by replacing trailing submatrix $\T(m-1\! :\! m,m-1\! :\! m)$ of $\T$ by better choices, obtained by inspecting the determinants of $2\times 2$ submatrices of $\R(m-1\! :\! m,m+1\! :\! n)$ and moving the corresponding columns upfront. 
}
\end{remark}
   

\begin{example}\label{EX1}
	{\em 
	We illustrate the difference in the values of  $\C=\|(\SO^T \mathsf{U})^{-1}\|_2$
	computed by \deim and \dime using $200$ randomly generated orthonormal matrices of size $10000 \times 100$. The \dime selection not only enjoys better upper bound, but it also in most cases provides smaller actual value of
	$\C$, as seen on Figure \ref{FIG1}. 
	It is interesting to note that all \dime values of
	$\|(\SO^T\U)^{-1}\|_2$ are below $100$, sustaining the conjectured
	bound  $\|(\SO_{\star}^T\U)^{-1}\|_2\leq\sqrt{n}$ for \textcolor{black}{the} volume maximizing scheme. 
	\textcolor{black}{On the other hand, \deim breaches the $\sqrt{n}$ upper bound in most of the trials, indicating less optimal selection with respect to the volume maximizing criterion. In the case of matrices specially constructed  to exhibit large pivot growth, the value of $\C$ in both methods may exceed $\sqrt{n}$, but not with a big factor.}
	We also compare \deim and \dime using a $2048\times 100$ basis for a FitzHugh-Naguma system, analyzed in detail in \S \ref{ex:fn}. The value of $\C$ for \dime is fixed at $2.6878e+01 < \sqrt{2048}$, independent of the orthogonal changes of the basis.

\begin{figure}[hhh]
	\begin{center}
		\includegraphics[width=2.3in,height=1.47in]{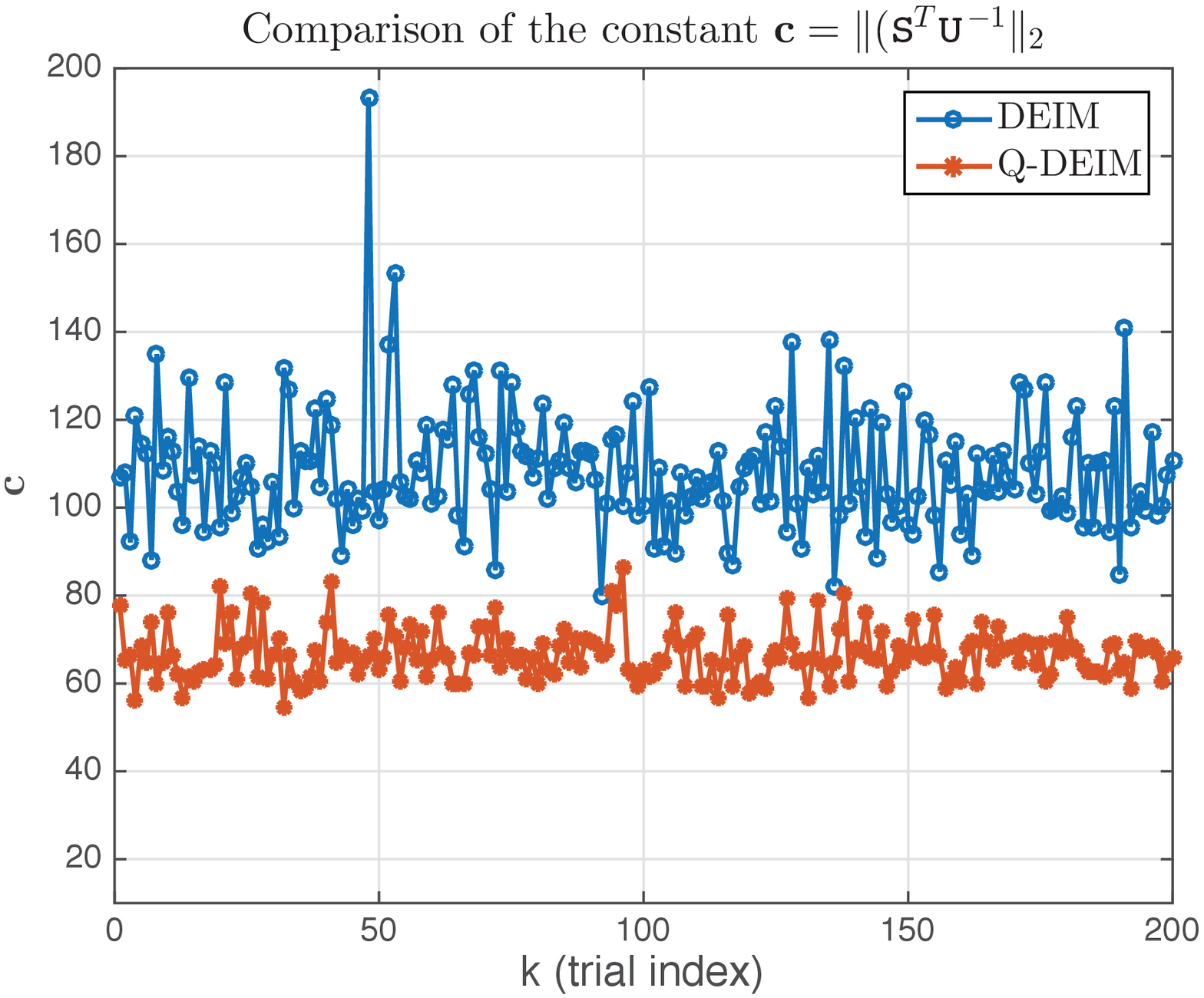} \hspace{1ex}
		\includegraphics[width=2.3in,height=1.47in]{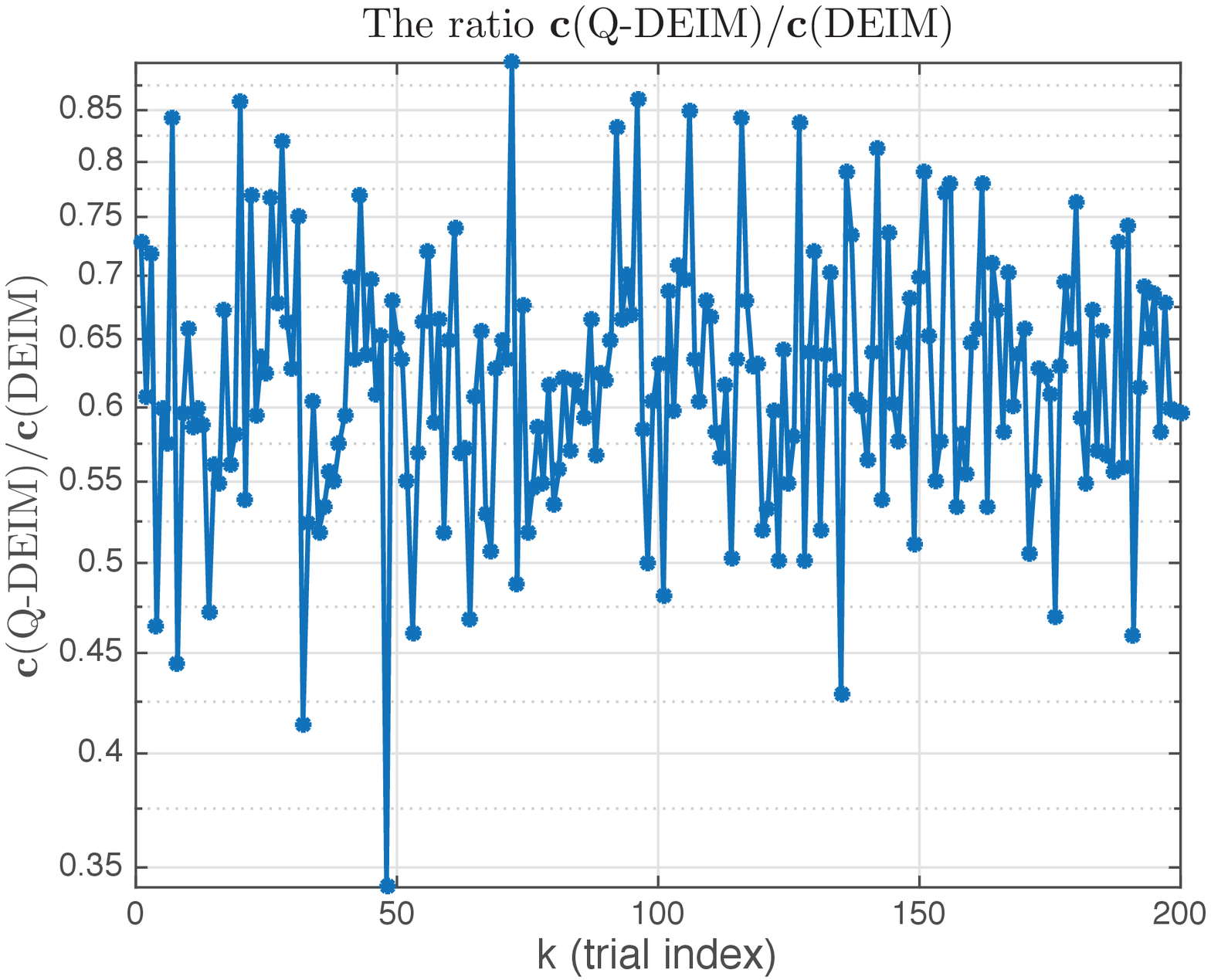}  \\ \vspace{1ex}
	         \includegraphics[width=2.3in,height=1.47in]{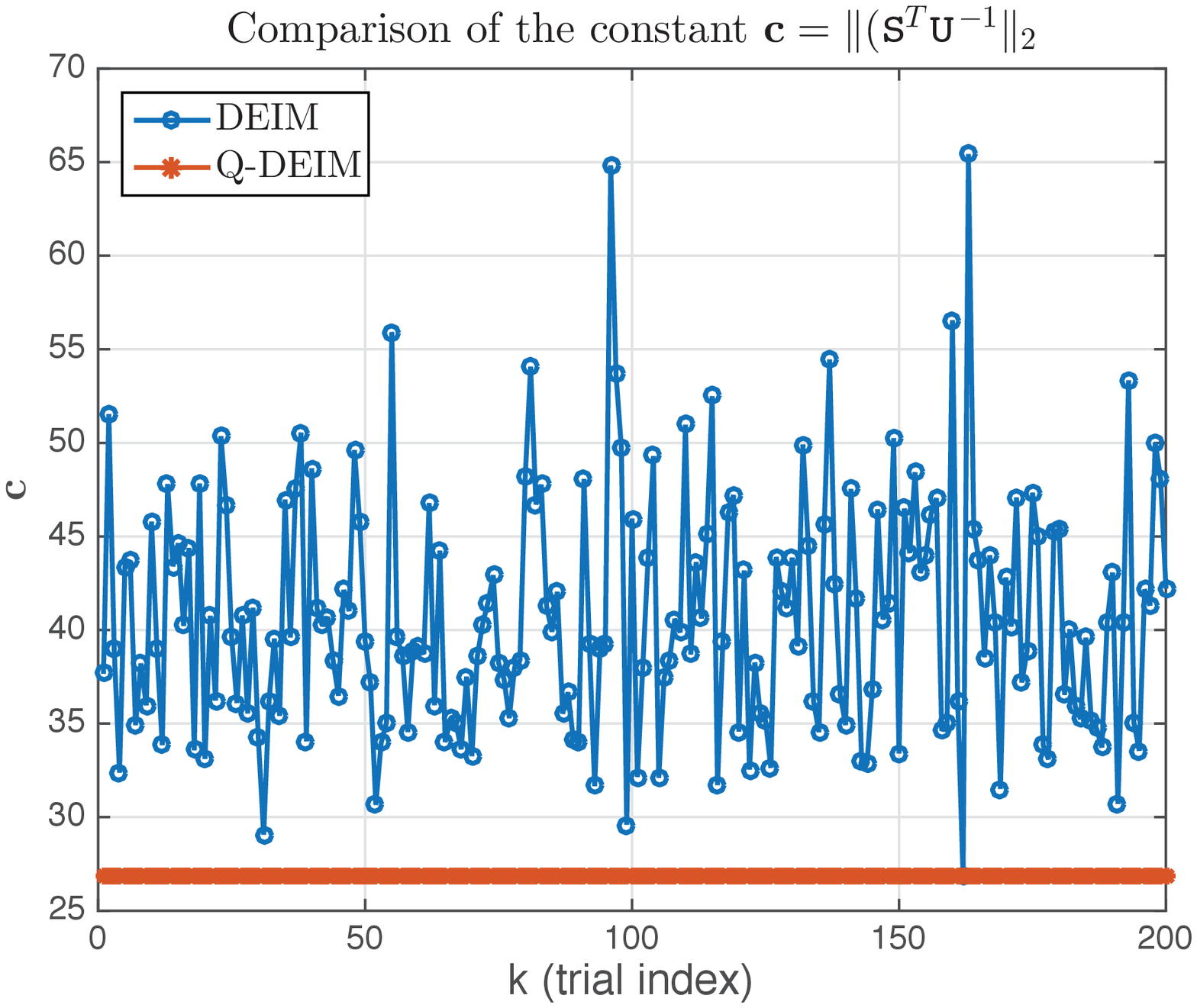}  \hspace{1ex}
		\includegraphics[width=2.3in,height=1.47in]{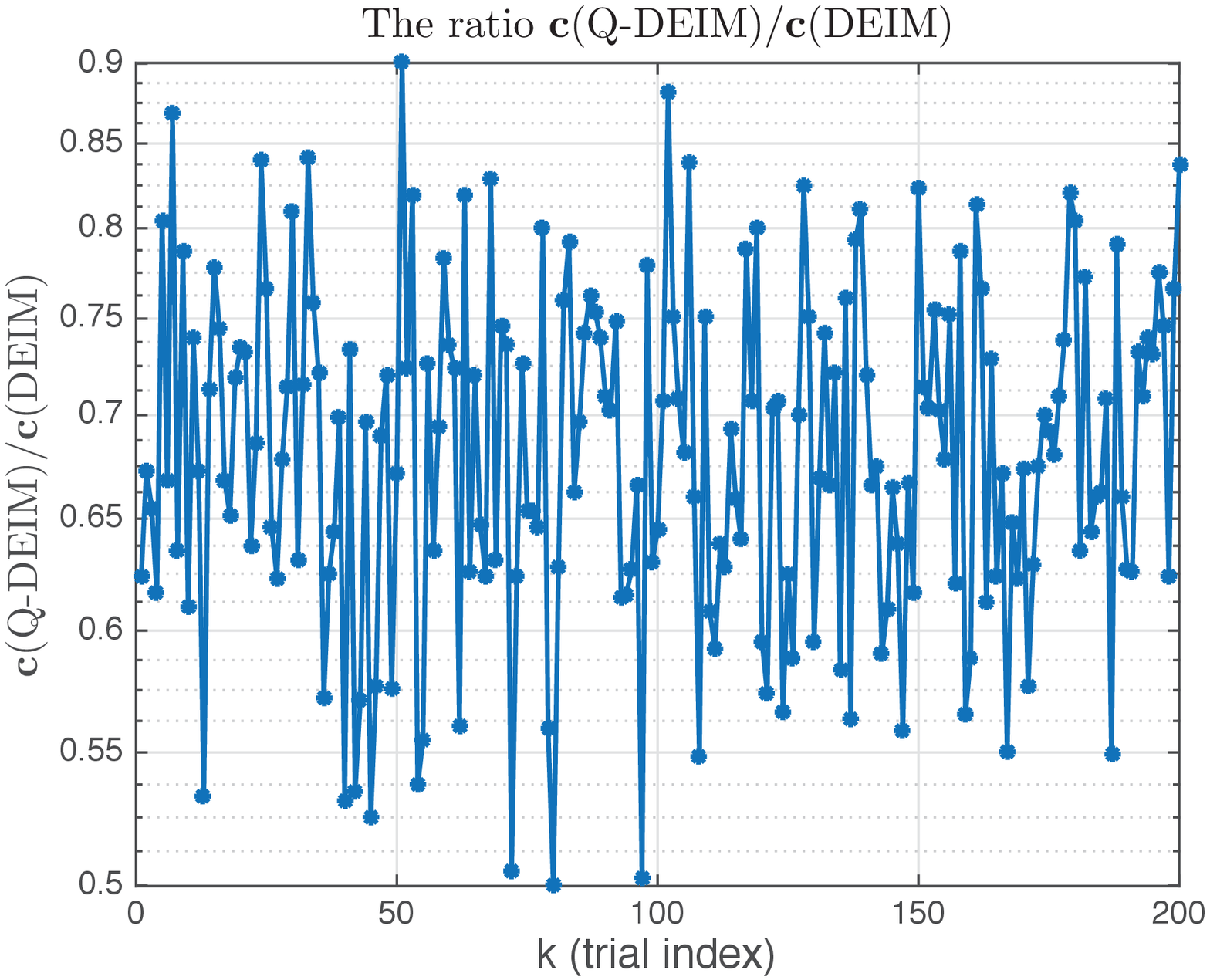}
	\end{center}
	\caption{\label{FIG1} (Example \ref{EX1}) Comparison of the value $\C=\|(\SO^T \mathsf{U})^{-1}\|_2$ in \deim and \emph{\dime}. The first row: The comparison using 200 random orthonormal matrices of size $10000 \times 100$.
	The second row: $200$ random changes of a \deim orthonormal basis $\U$ of size $2048 \times 100$, computed from simulation of the FitzHugh-Naguma system, see \S \ref{ex:fn}. The basis changes are obtained by post-multiplication by random $100\times 100$ real orthogonal matrices (uniformly distributed in the Haar measure).}
\end{figure} 

}	
\end{example}
%


\subsubsection{Implementation details}
		In terms of the row selection from $\U$, the actual computation used in  Theorem \ref{Theorem:2mbound} is an LQ factorization of $\U$ with row pivoting.
		The transposition and  QR with column pivoting is used only for convenience and due to the availability of software implementations.
		A Householder based QR factorization of a fat $m\times n$ matrix runs with complexity $O(m^2 n)$, similar to the complexity of \deim.		
		{LAPACK} \cite{LAPACK} based software tools use \textcolor{black}{the} optimized, BLAS 3 based and robust \cite{drmac-bujanovic-2008} function {\tt xGEQP3}. 
		Other pivoting strategies are possible, such as in {\tt xGEQPX, xGEQPY} in 
		\cite{bischof-q-orti-RRQR-1998-TOMS782}, \cite{bischof-q-orti-RRQR-1998}. On  parallel computing machinery, our new approach uses \textcolor{black}{the} best available QR code with column pivoting; e.g. {\tt PxGEQPF} from {ScaLAPACK} \cite{ScaLAPACK}. 
		
		As an illustration of \textcolor{black}{the} simplicity at which we get high performance computation of a good selection operator, and to make a case for \dime, we briefly describe a MATLAB implementation. 
		Using the notation of Theorem \ref{Theorem:2mbound}, we have
		\begin{equation}\label{eq:dime-matrix}
		\U = \Pi \begin{pmatrix}
		\T^* \cr \K^* \end{pmatrix}\mathsf{Q}^* ,\;\;\SO^T\U = \T^* \mathsf{Q}^*,\;\;\mbox{and thus}\;\; \M\equiv \U(\SO^T\U)^{-1} = \Pi \begin{pmatrix} 
		\Id_m \cr \K^*\T^{-*} 
		\end{pmatrix} .
		\end{equation}

 
\noindent The computation $\T^{-1}\K = \breve{\T}^{-1}({\mathsf D}^{-1}\K)$ by a triangular solver (e.g. the backslash \texttt{\textbackslash} or \texttt{linsolve()} in MATLAB) is numerically stable as $\breve{\T}$ is well conditioned and
$\max_{ij}|({\mathsf D}^{-1}\K)_{ij}|\leq 1$. The explicitly set identity matrix $\Id_m$ in (\ref{eq:dime-matrix})
guarantees that the selected entries of a vector $f$ will be exactly interpolated
when $\M$ is computed as in (\ref{eq:dime-matrix}). If $\M$ is computed as
$\widetilde{\M}=computed(\U/(\SO^T\U))$ (e.g. by MATLAB's slash)
then $\SO^T\widetilde{\M} = \Id_m + (\epsilon_{ij})_{m\times m}$,
with all $\epsilon_{ij}$ at roundoff level. If $s_1,\ldots, s_m$ are the interpolation indices selected by $\SO$, then checking the interpolation for $f\in\mathbb{{R}}^n$ yields
$$
(\SO^T \widetilde{\M} \SO^T f)_i = f_{s_i}(1+\epsilon_{ii}) + \sum_{j\neq i}
f_{s_j}\epsilon_{ij} ,\;\;i=1,\ldots, m,
$$
revealing an undesirable  pollution of $f_{s_i}$, in particular if $\max_{s_j\neq s_i}|f_{s_j}|\gg |f_{s_i}|$.  

\vspace{-2mm} 
\begin{lstlisting}
function [ S, M ] = q_deim( U ) ;
% Input  : U n-by-m with orthonormal columns
% Output : S selection of m row indices with guaranteed upper bound
%          norm(inv(U(S,:))) <= sqrt(n-m+1) * O(2^m).
%        : M the matrix U*inv(U(S,:)); 
%          The Q-DEIM projection of an n-by-1 vector f is M*f(S).
% Coded by Zlatko Drmac, April 2015.
[n,m] = size(U) ; 
if nargout == 1 
  [~,~,P] = qr(U','vector') ; S = P(1:m)  ;
else
  [Q,R,P] = qr(U','vector') ; S = P(1:m)  ; 
  M = [eye(m) ; (R(:,1:m)\R(:,m+1:n))']   ;
  Pinverse(P) = 1 : n ; M = M(Pinverse,:) ; 
end
end
\end{lstlisting} 

\subsubsection{\deim and LU with partial pivoting}\label{SS:DEIM-LUPP}
It has been known, at least to experts, that \deim is a variation of Gaussian elimination;  Sorensen \textcolor{black}{\cite{dan,dansiam10}} called it a \emph{pivoted LU without replacement}. 
Recently, \cite{antil2013two} proposed to replace Step 7 of Algorithm \ref{ALG:DEIM}, $\U_j = \begin{pmatrix} \U_{j-1} & u_j \end{pmatrix}$,
by $\widehat{\U}_j = (\begin{matrix}\widehat{\U}_{j-1} & \widehat{r} \end{matrix})$ (To make the distinction clear, we denote the new variable by $\widehat{\U}$ and we use $\widehat{\U}_j$ 
to denote the matrix $\widehat{\U}(:,1\! :\! j)$  at the end of the $j$th step.). In other words, the basis $\widehat\U_{j-1}$ is updated by adding the current residual vector. Assume that at step $j$, 
$\widehat{\U}_{j-1}=\U_{j-1} \G_{j-1}$, where $\G_{j-1}$ is unit upper triangular and that both computations have the same selection $\SO_{j-1}$ (this holds at $j=2$). Then the residual $\widehat{r}=u_j - \widehat{\U}_{j-1}(\SO_{j-1}^T\widehat{\U}_{j-1})^{-1}\SO_{j-1}^T u_j$ is
easily shown to be the same as $r$ in Algorithm \ref{ALG:DEIM}. Hence, the updated $\SO_j$ will be the same, and $\widehat{\U}_{j}\equiv (\begin{matrix} \widehat{\U}_{j-1} & \widehat{r}\end{matrix})=\begin{pmatrix} \U_{j-1} & u_j \end{pmatrix} \G_{j}$, with an updated unit upper triangular $\G_{j}$.

 Note that $\SO_j^T \widehat{\U}_j=\begin{pmatrix} \SO_{j-1} & e_{p_j} \end{pmatrix}^T (\begin{matrix} \widehat{\U}_{j-1} & \widehat{r} \end{matrix})$ is  lower  triangular at each step: The upper triangular part of the last column of this product is $\SO_{j-1}^T \widehat{r}$, which is zero by the definition of $\widehat{r}$. Therefore, $\SO^T \widehat{\U} =\SO^T \U \G = \Z$ with a lower trapezoidal $\Z$ and a unit upper triangular $\G$. Hence, \deim (with replacement) is the same as a row pivoted LU decomposition. This connection  might help understanding why \deim behaves much better than the theoretical upper bound would suggest.  Similar to the discussion in Remark \ref{rem:volume}, \deim is applying a locally greedy search to maximize the volume of $\SO^T \U$: Because $\widehat{\U}=\U \G$ with a unit upper triangular $\G$, $\mathrm{det}(
\SO^T\widehat{\U}) = \mathrm{det}(\U)=\prod_{j=1}^m \Z_{jj}$, and the
$\Z_{jj}$'s are results of a greedy search for maxima.


The modified update yielding $\widehat{\U}$ reduces the computational complexity of $\widehat\U(\SO^T \widehat\U)^{-1}$ to  $O(nm^2)$ down from $O(nm^2)$+ $O(m^3)$ because the LU decomposition of $\SO^T \widehat\U$ is no longer necessary. However, the claim in \cite{antil2013two} that  Algorithm \ref{ALG:DEIM} contains  
$O(m^4)$   complexity  is misleading because a practical implementation of Step 5,  $\SO_{j-1}^T \U_j z = \SO_{j-1}^T u_j$,  will not 
compute the LU decomposition of $\SO_{j-1}^T \U_{j-1}$ from scratch in each step. Instead,  
 it will exploit the fact that, in the $j^{\rm th}$ step, $\SO_{j-1}^T \U_{j-1}$ is changed by appending only a new row and a column and will \emph{update} the LU decomposition with complexity $O(j^2)$; thus making the total complexity of Algorithm 
 \ref{ALG:DEIM} of $O(nm^2)$+ $O(m^3)$. 
   

This connection also suggests  using a rank revealing LU decomposition of $\U$ with complete pivoting and taking
the indices of the first $m$ pivoted rows as the \deim indices. The bound on $\C$ changes only by a factor
originating from the inverse of the unit upper triangular LU factor, which, as a consequence of complete pivoting, is bounded by $O(2^m)$. On the other hand, in the case of partial pivoting like in \deim, the bound  
is rather pessimistic. However, as expected, \deim behaves much better in practice, as partial LU hardly exhibits the worst case growth scenario. 
\textcolor{black}{It is an interesting challenge to determine what
orthonormal basis of the range of $\U$ is best for the performances
of \deim.}

\subsection{Model Reduction Examples}\label{S=examples}
In this section, we test the performance of the new selection procedure on {two} model reduction benchmark problems.

\subsubsection{The FitzHugh-Naguma (F--N) System}  \label{ex:fn}
The F--N system, a simplified version of the Hodgkin--Huxley model,  arises in 
modeling the activation and deactivation dynamics of a spiking neuron. This example is borrowed from
\cite{DEIM} and we follow their description of the model, including their notation and  parameter selection. 
Let $v$ and $w$ denote, respectively, the voltage and recovery of voltage. Also, let $x \in [0,L]$ and $t\geq0$. Then,
the underlying dynamics are described by the coupled  nonlinear PDEs
\begin{eqnarray}
\varepsilon v_t(x,t) & = &  \varepsilon^2 v_{xx}(x,t) + f(v(x,t))-w(x,t) + c \\
w_t(x,t) & = & b v(x,t) - \gamma w(x,t) +c
\end{eqnarray}
with the nonlinearity appearing as $f(v) = v(v-0.1)(1-v)$ and 
with the initial and boundary conditions 
$$
\begin{array}{ccc}
v(x,0) = 0, & w(x,0) =0, &  x  \in [0,L],\\
v_x(0,t)  =  -i_0(t), & v_x(L,t)  =  0, & t  \geq  0,
\end{array}
$$
where the model  parameters are chosen as $L=1$, $\varepsilon = 0.015$, $b=0.5$, $\gamma = 2$, $c=0.05$ and the 
stimulus $i_0(t) = 50000t^3 e^{-15t}$. A finite difference discretization leads to a system of the form {(\ref{nlfom})} with system dimension $n=2048$. A time-domain simulation with equally spaced points for $t=[0,8]$ leads to $N=100$ state and nonlinear snapshots. Following \cite{DEIM}, we choose $r=m=5$ and perform model reduction using both  \deim selection procedures. We simulate  both reduced models, collect reduced-order snapshots $\Xs_{\deim}$ and $\Xs_{\dime}$. Then, to measure the error in model reduction, we lift these two snapshots back  to the original dimension, i.e., we compute $V\Xs_{\deim}$ and  $V\Xs_{\dime}$ where 
$V$ is the POD basis for model reduction, and measure their distance (in the relative Frobenius-norm) from the original snapshot $\Xs$. Let 
$\epsilon_{\deim} = \frac{\|\Xs - V \Xs_{\deim}\|_F}{\|\Xs\|_F} $ and 
$\epsilon_{\dime} = \frac{\|\Xs - V \Xs_{\dime}\|_F}{\|\Xs\|_F} $
denote the resulting errors. For $r=m=5$, we obtain 
$\epsilon_{\deim} = 3.500673 \times 10^{-2}$ and 
$\epsilon_\dime = 3.467286 \times 10^{-2}$.

To illustrate that this is the \textcolor{black}{usual behavior}, i.e., the \dime selection performs as well as the original \deim selection,   we test other $r=m$ values as well:
$$
\begin{array}{ccc}
r=m=4: & \epsilon_{\deim} = 4.291788 \times 10^{-2}, & \epsilon_{\dime} =3.446203 \times 10^{-2}   \\
r=m=6: & \epsilon_{\deim} = 3.300680\times 10^{-2}, & \epsilon_{\dime} = 3.260097\times 10^{-2}   \\
r=m=7: & \epsilon_{\deim} = 2.998979 \times 10^{-2}, & \epsilon_{\dime} =  3.010827 \times 10^{-2}.
\end{array}
$$ 
To better illustrate the comparison, we measure, in relative $2$-norm, how accurately each entry of $x(t)$ is reconstructed with \deim and \dime. Therefore, we measure the relative $2$-norm distance between the $k^{\rm th}$ rows  \textcolor{black}{of}
the original snapshot matrix $\Xs$ and those of the reconstructed ones $V\Xs_{\deim}$ and 
$V\Xs_{\dime}$. The results are shown in Figure \ref{fig:4646}\textcolor{blue}{,} once more\textcolor{blue}{,} illustrating that both procedures perform equally well.
\begin{figure}
	\begin{center}
		\includegraphics[scale=0.7]{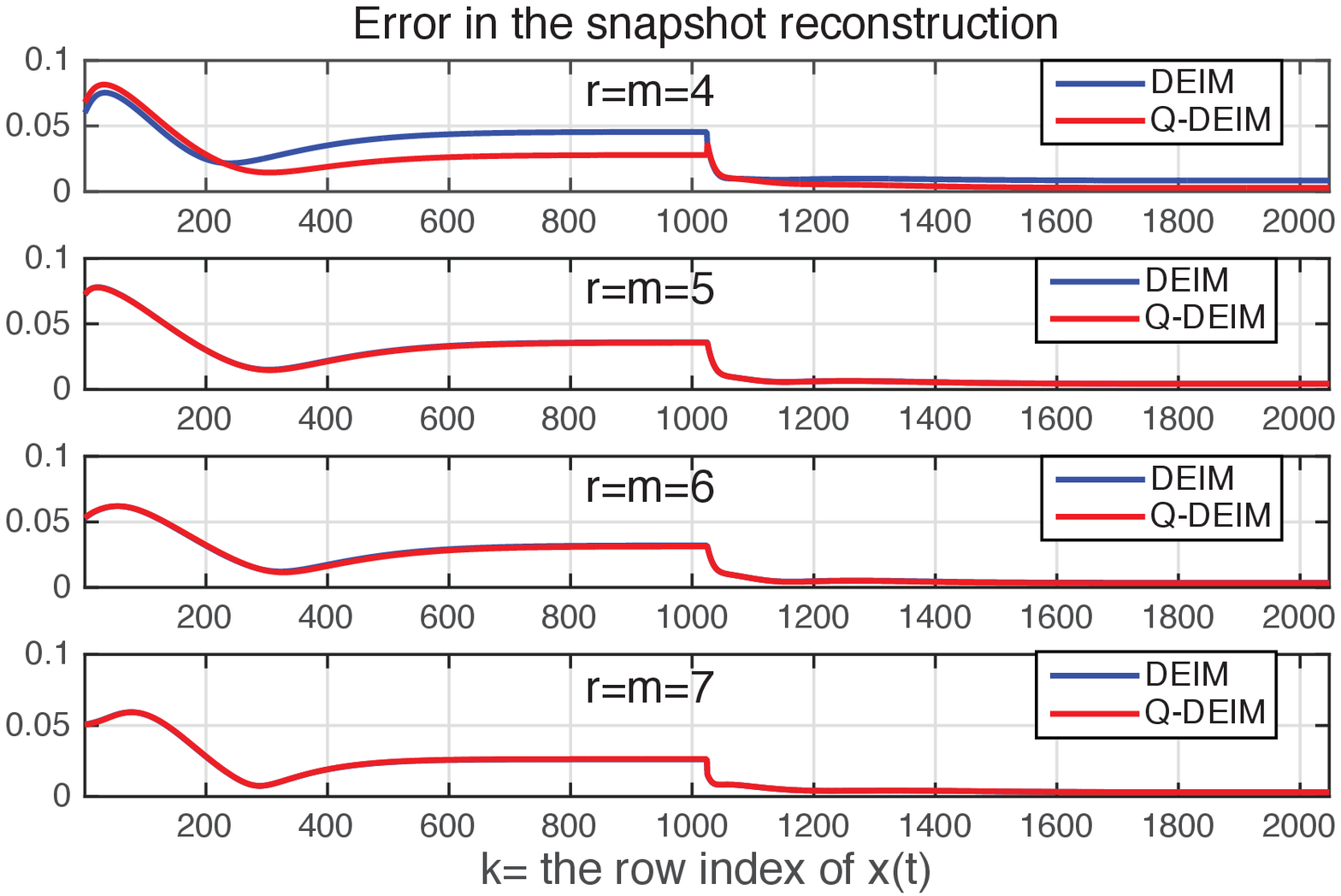} 
	\end{center}
	\caption{Example \ref{ex:fn}. \label{fig:4646} Comparison of the relative error in the snapshot reconstruction for the F-N model for different  $r$ and $m$ values. The horizontal axis variable $k$ corresponds to the $k^{\rm rh}$ row of $x(t)$ for which the relative error is computed.
	}
\end{figure} 

\subsubsection{Nonlinear RC Model} \label{ex:rc}
This is a model of  nonlinear RC-ladder circuit  \cite{Chen1999,bai2006projection,benner2015two}\footnote{The model can be downloaded from  Max Planck Institute Model Reduction Wiki page at  {\tt http://morwiki.mpi-magdeburg.mpg.de/morwiki/index.php/Nonlinear$_{-}$RC$_{-}$Ladder}.}, another benchmark example for model reduction.
\textcolor{black}{The} nonlinearity results from resistors that are in a parallel connection with diodes; the diode I-V characteristics have the nonlinearity $i_D = e^{40 v_D}-1$ 
where $i_D$ is the current through the diode and $v_D$ is the voltage across it. 
The input is the current source entering at node $1$. 

We take $n=1000$, i.e., connect $1000$ such ladders and excite the system using 
the exponential forcing  $\mathbf{g}(t) = e^{-t}$. A numerical simulation over $t=[0,7]$ seconds results in $1425$ POD snapshots $x_i$ and $1425$ nonlinear (\deim) snapshots $\f(x_i)$. Decay of the POD and \deim singular values are shown in the left-hand side plot of Figure \ref{fig:rlc1}. Based on this decay, we pick $r=m=10$ and apply POD with both \deim and \dime selections. In this example, the voltage at node $1$, i.e., the first component of $x(t)$ \textcolor{black}{(denoted by $\xi_1(t)$)}, is the quantity of interest and we measure how both reduced models approximate \textcolor{black}{$\xi_1(t)$}.  As shown on Figure \ref{fig:rlc1},
both methods perform extremely well; the reduced-model quantities are virtually indistinguishable from the original.

As in the previous example, we compute the reconstruction errors due to \deim and \dime, and obtain $\epsilon_{\deim} =8.603826 \times 10^{-3} $ and 
$\epsilon_{\dime} = 6.07172 \times 10^{-3}$; once again high accuracy for both models. The reconstruction errors in \textcolor{black}{$\xi_1(t)$}  due to \deim and \dime are
$1.28183\times 10^{-4} $  and $7.783045 \times 10^{-5}$ respectively. Once we increase
$r=m=10$ to $r=m=20$, both reduced models become even more accurate with 
$\epsilon_{\deim} = 1.970500 \times 10^{-4}$ and   $\epsilon_{\dime} =  1.931018 \times 10^{-4}$. The reconstruction errors in \textcolor{black}{$\xi_1(t)$}  are now 
$3.209967\times 10^{-5} $ for \deim  and $3.238549 \times 10^{-5}$ for \dime.

\begin{figure}[hhh]
	\begin{center}
		\includegraphics[width=2.3in,height=2.0in]{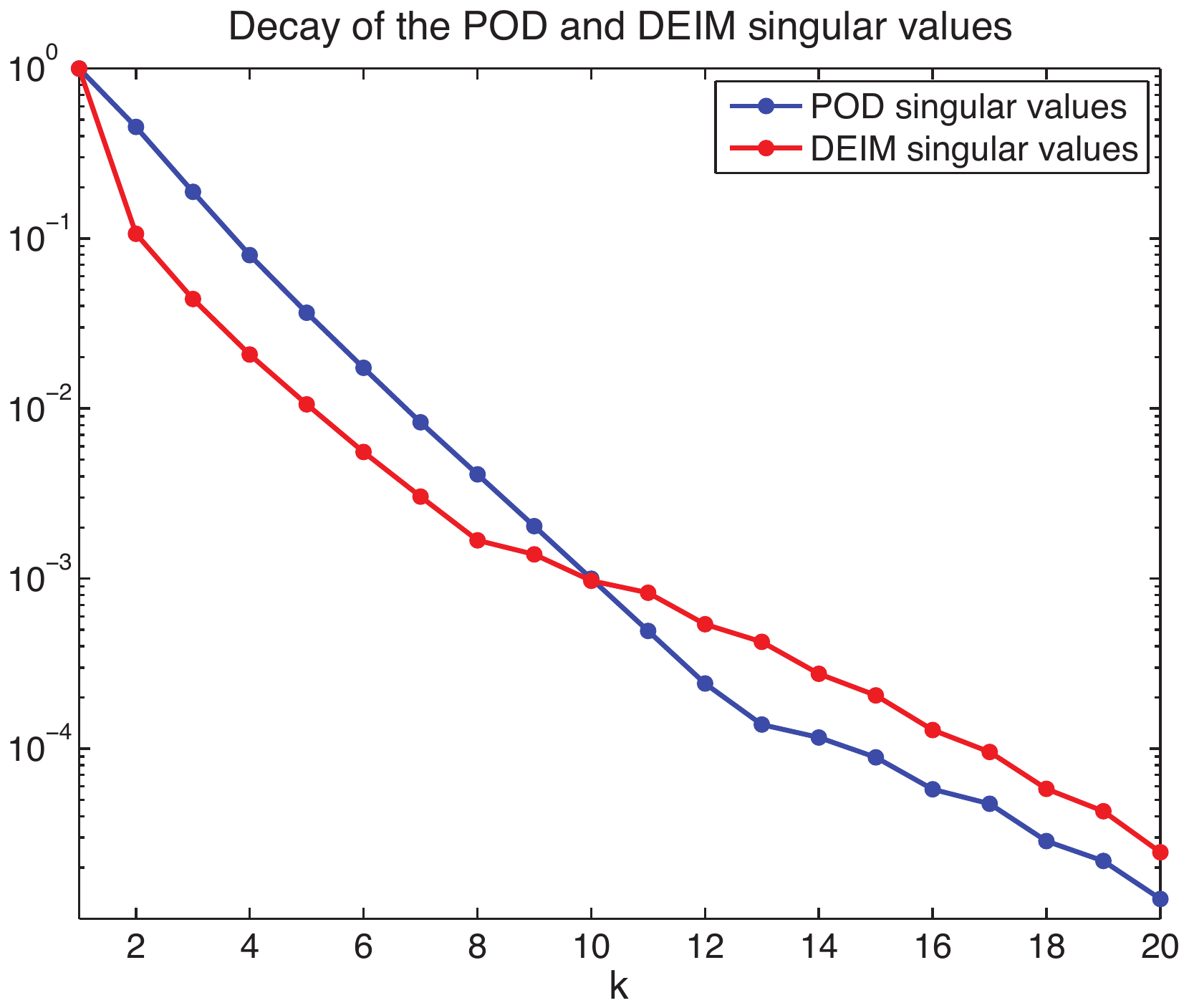}
		\includegraphics[width=2.3in,height=2.0in]{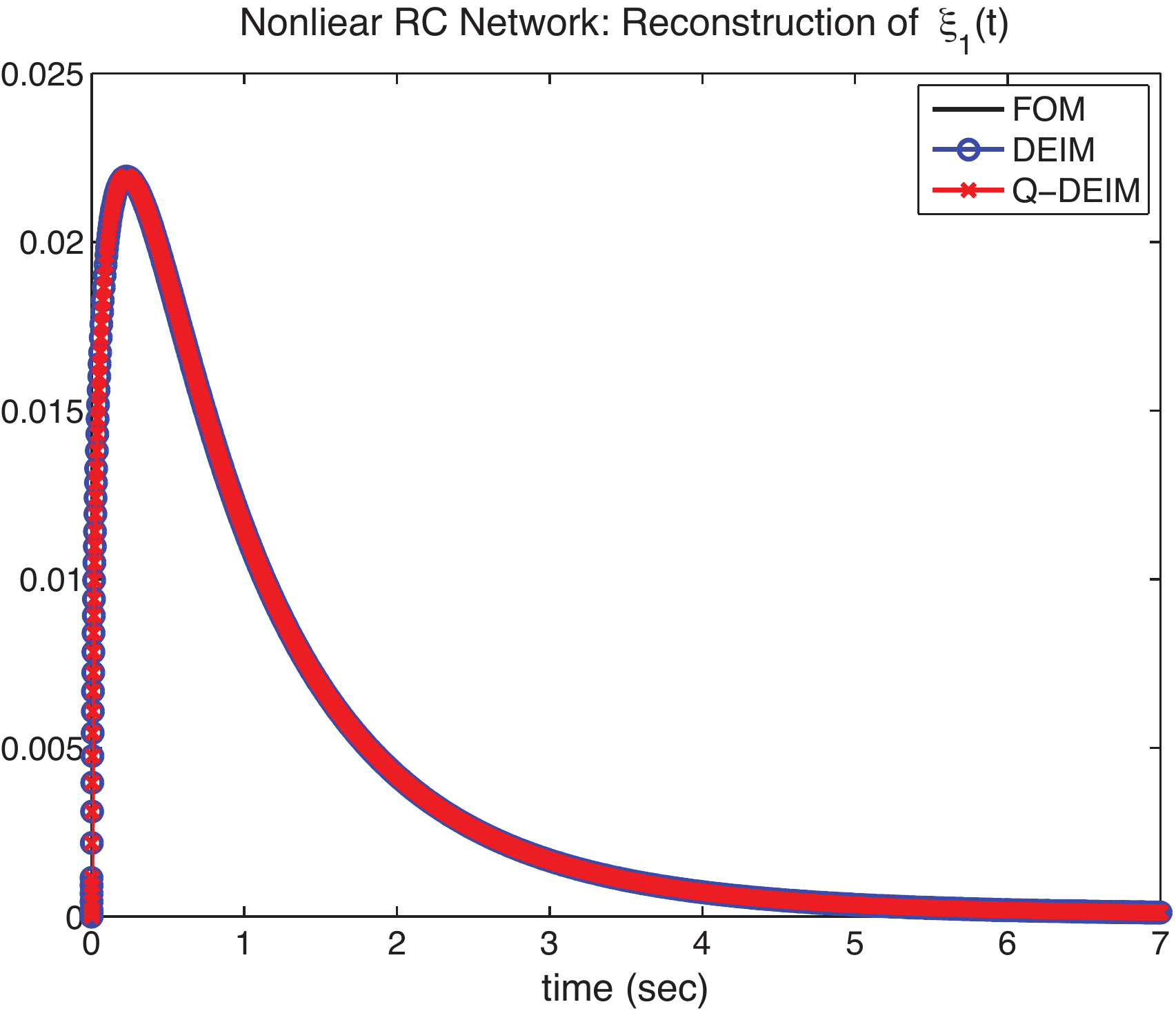}
	\end{center}
	
	\caption{ \label{fig:rlc1} (Example \ref{ex:rc}) The left-plot shows the decay of the POD and \deim singular values. The right-plot shows the reconstruction accuracy of $\xi_1(t)$ by both selection methods.}
\end{figure} 	

To further illustrate the dependence of the error on the reduced dimension, we test both \deim and \dime for $1\leq r=m \leq 20$. The results  depicted in Figure \ref{fig:vary_r_m} below confirm the earlier observations.
\begin{figure}[hhh]
	\begin{center}
		\includegraphics[scale=0.4]{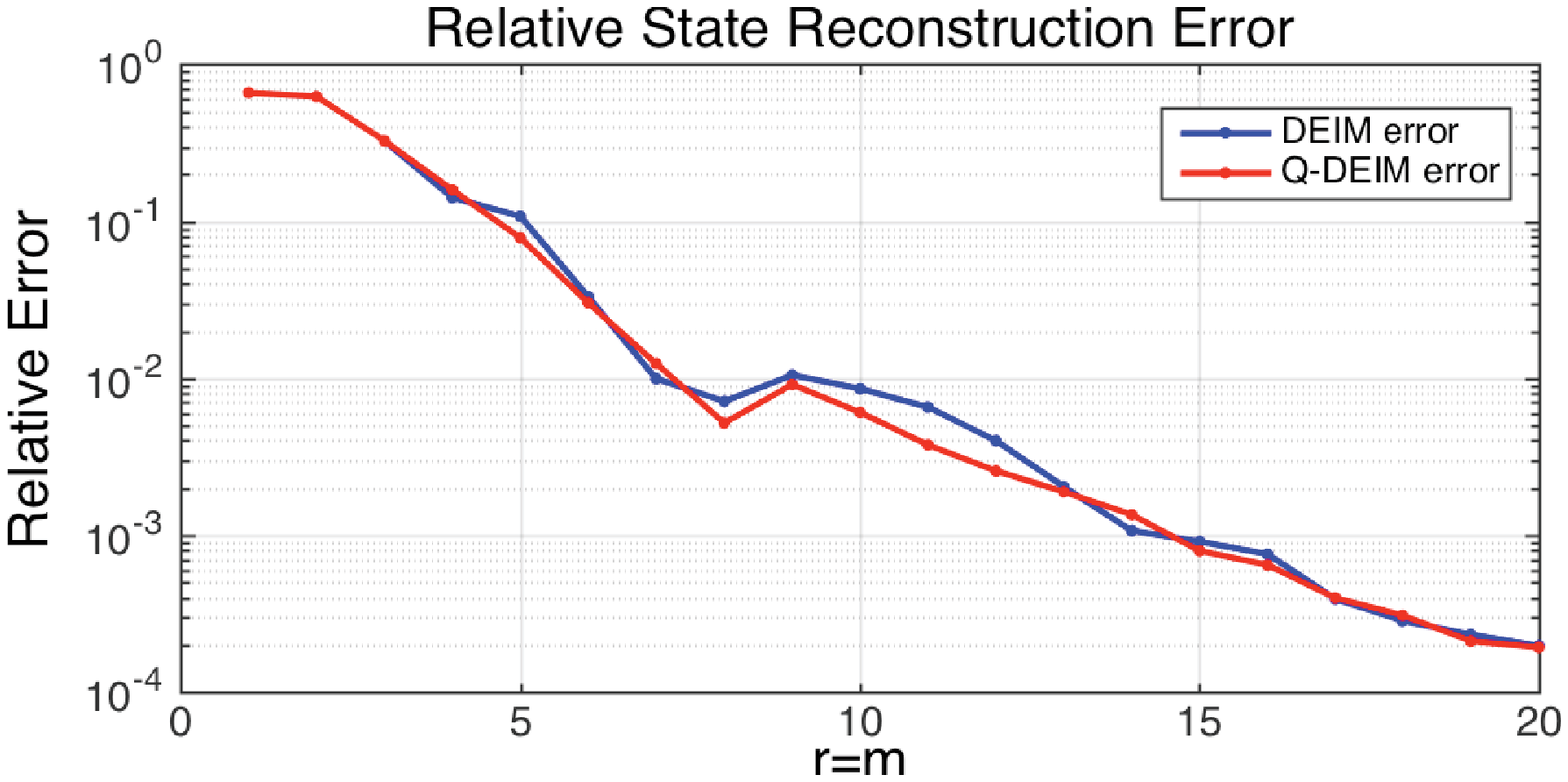}
\end{center}
	
	\caption{ \label{fig:vary_r_m} (Example \ref{ex:rc}) Relative errors $\epsilon_{\deim}$ and $\epsilon_{\dime}$ for 
	varying $r$ and $m$ values.}
\end{figure} 	

The two other excitation selections suggested for this model are 
$\mathbf{g}(t) = \sin(2 \pi 50t)$ and   $\mathbf{g}(t) = \sin(2 \pi 1000t)$ \cite{condon2004model}. For these two inputs, for all the $r$ and $m$  combinations we have tried, \deim and \dime have returned exactly the same selection matrix $\SO$; resulting in exactly same reduced model that is very accurate, with relative errors of $O(10^{-3})$ even with $r=m=5$. For brevity, we omit the resulting figures.

\section{Using restricted/randomized basis information}\label{SS=dimer}
The framework introduced in \S \ref{S=dime} \textcolor{black}{allows} various modifications. Here 
we describe one, introducing \dimer, a version of \dime that works only on a random selection of the rows of $\U$. This 
introduces the techniques of randomized sampling in the theory and practice of \deim, but it also provides a new approach to sampling from orthonormal matrices \cite{Ipsen:2014:ECS}.

If the dimension $n$ is large, it would be advantageous to determine a good
selection operator $\SO$ in a way to reduce the $O(m^2 n)$ factor in the complexity of
the algorithm.  From the proof of Theorem \ref{Theorem:2mbound} it follows that, for the quality of the \deim projection, we
only need to ensure that the upper triangular matrix $\T=\R(1\! :\! m,1\! :\! m)$ has small
inverse. But $\T$ is the pivoted QR triangular factor of certain columns of $\W$,
and our initial task is to find their indices; we have no interest in the QR
factorization as such.  This immediately suggests that we may attempt to find such indices using only a small selection of the columns of $\W$ (i.e. of the rows of $\U$).\footnote{In another situation, e.g. in a case of gappy POD approximation,  we may want to avoid some rows of $\U$ because e.g. they correspond to spatial coordinates with corrupted or missing information. 
	}

The following scheme depicts the main idea: suppose we randomly sample $k\geq m$ 
columns of $\W$ (marked by $\begin{smallmatrix}\uparrow\cr \ast\end{smallmatrix}$) and assemble them in a local working $m\times k$ array $\Lw$. 
{\small
	\begin{equation}\label{eq:dimer1}
	\bordermatrix{
		& \textcolor{blue}{\Uparrow}   &  \uparrow   &   &      & \uparrow  &    & \uparrow   &   & \uparrow   &  \textcolor{blue}{\Uparrow} &   &   & \textcolor{blue}{\Uparrow}  &   & \uparrow  &   & \uparrow  &   &     \textcolor{blue}{\Uparrow} \cr
		&  \textcolor{blue}{\star} &  \ast  & . &   .  & \ast &  . & \ast  & . &  \ast & \textcolor{blue}{\star} & . & . & \textcolor{blue}{\star} & . & \ast & . & \ast & .  & \textcolor{blue}{\star} \cr
		&  \textcolor{blue}{\star} &  \ast  & . &   .  & \ast &  . & \ast  & . &  \ast & \textcolor{blue}{\star} & . & . & \textcolor{blue}{\star} & . & \ast & . & \ast & .  & \textcolor{blue}{\star} \cr
		&  \textcolor{blue}{\star} &  \ast  & . &   .  & \ast &  . & \ast  & . &  \ast & \textcolor{blue}{\star} & . & . & \textcolor{blue}{\star} & . & \ast & . & \ast & .  & \textcolor{blue}{\star} \cr
		&  \textcolor{blue}{\star} &  \ast  & . &   .  & \ast &  . & \ast  & . &  \ast & \textcolor{blue}{\star} & . & . & \textcolor{blue}{\star} & . & \ast & . & \ast & .  & \textcolor{blue}{\star} \cr
	}	
	\mapsto \overbrace{\begin{pmatrix}
		*  &  *  &  *  &  *  &  *  & * \cr
		*  &  *  &  *  &  *  &  *  & * \cr
		*  &  *  &  *  &  *  &  *  & * \cr
		*  &  *  &  *  &  *  &  *  & * \cr
		\end{pmatrix}}^{{\displaystyle \Lw}}
	\end{equation}
}
Then we attempt QR with column pivoting on $\Lw$ with a built-in  Incremental Condition Estimator (ICE) that at any step $j$ very efficiently estimates the norm of the inverse of the thus far constructed part of the triangular factor $\Lw(1\! :\! j,1\!:\! j)$.\footnote{Here we assume that the triangular factor overwrites the corresponding part of the array $\Lw$.} 

This is done as follows: Suppose the first $j-1$ steps have been successful and the computed triangular factor ($\left(\begin{smallmatrix} \textcolor{green}{\ast} & \textcolor{green}{\ast} \cr 0 &  \textcolor{green}{\ast}\end{smallmatrix}\right)$ in (\ref{eq:dimer2})) is well conditioned. The corresponding global column indices (in $\W$) that
correspond to these columns in $\Lw$ are stored and removed from the active set of indices from which random selection is made. All Householder reflectors used are accumulated in a $m\times m$ matrix $\Q$.

In the $j$th step, a new pivot column is selected and
swapped to be the $j$th one, and the single Householder reflector is applied only to that
column to annihilate its positions $j+1$ to $m$, and to compute the $(j,j)$th position,
($\textcolor{red}{\circledast}$ in (\ref{eq:dimer2})). At this moment we have all the ingredients to compute the value $\gamma_j=\|\Lw(1\! :\! j,1\!:\! j)^{-1}\|$ (marked as $\left(\begin{smallmatrix} \textcolor{green}{\ast} & \textcolor{green}{\ast} & \textcolor{red}{\times} \cr 0 &  \textcolor{green}{\ast} & \textcolor{red}{\times} \cr 0 & 0 & \textcolor{red}{\circledast}\end{smallmatrix}\right)$ in (\ref{eq:dimer2})).
If $\gamma_j$ is below given threshold, the factorization continues by accepting the pivotal column, completing the $j$th step and looking for the next pivot. If not, it means that the $(j,j)$th position,
($\textcolor{red}{\circledast}$ in (\ref{eq:dimer2})) is too small, and, due to pivoting,  that all entries in the active submatrix of $\Lw$ ($\textcolor{red}{\odot}$ in (\ref{eq:dimer2})) are also small. In that case, the columns $j$ to $k$ in $\Lw$ are useless for our purposes and we discard them and draw new $k-j+1$ columns from the active set of columns of $\W$ ($\begin{smallmatrix}\textcolor{blue}{\Uparrow}\cr\textcolor{blue}{\star}\end{smallmatrix}$ in (\ref{eq:dimer1})).
Before using newly selected columns as a part of $\Lw$, we need to update them by applying unitary matrix $\Q$ which contains accumulated all previous transformations. The resulting new columns in $\Lw$ ($\textcolor{magenta}{\star}$ in (\ref{eq:dimer2})) can now participate in pivoting. 

{\small
\begin{equation}\label{eq:dimer2}
\begin{pmatrix}
	\textcolor{green}{\ast} & \textcolor{green}{\ast} &  \textcolor{red}{\times}  &   \times  &   \times  & \times \cr
	 0   & \textcolor{green}{\ast} &  \textcolor{red}{\times}   &   \times  &   \times  &  \times \cr
	 0   &   0  &  \textcolor{red}{\circledast}  &   \textcolor{red}{\odot}  &   \textcolor{red}{\odot}  &  \textcolor{red}{\odot} \cr
	 0   &   0  &  0  &   \textcolor{red}{\odot}  &   \textcolor{red}{\odot} &   \textcolor{red}{\odot} 
\end{pmatrix}
\rightsquigarrow
\begin{pmatrix}
	\textcolor{green}{\ast} & \textcolor{green}{\ast} &  \textcolor{magenta}{\star}  &   \textcolor{magenta}{\star}  &   \textcolor{magenta}{\star}  & \textcolor{magenta}{\star} \cr
	0   & \textcolor{green}{\ast} &   \textcolor{magenta}{\star}   &   \textcolor{magenta}{\star}  &   \textcolor{magenta}{\star}  &  \textcolor{magenta}{\star} \cr
	0   &   0  &  \textcolor{magenta}{\star}  &   \textcolor{magenta}{\star}  &   \textcolor{magenta}{\star}  &  \textcolor{magenta}{\star} \cr
	0   &   0  &  \textcolor{magenta}{\star}  &   \textcolor{magenta}{\star}  &   \textcolor{magenta}{\star} &   \textcolor{magenta}{\star} 
\end{pmatrix}
\rightsquigarrow
\begin{pmatrix}
	\textcolor{green}{\ast} & \textcolor{green}{\ast} &  \textcolor{green}{\ast}  &   \textcolor{magenta}{\star}  &   \textcolor{magenta}{\star}  & \textcolor{magenta}{\star} \cr
	0   & \textcolor{green}{\ast} &   \textcolor{green}{\ast}   &   \textcolor{magenta}{\star}  &   \textcolor{magenta}{\star}  &  \textcolor{magenta}{\star} \cr
	0   &   0  &  \textcolor{green}{\ast}  &   \textcolor{magenta}{\star}  &   \textcolor{magenta}{\star}  &  \textcolor{magenta}{\star} \cr
	0   &   0  &  0  &   \textcolor{magenta}{\star}  &   \textcolor{magenta}{\star} &   \textcolor{magenta}{\star} 
\end{pmatrix}
\end{equation}
}
At this point we may simply choose  to continue with the factorization, search for the next pivot column, swap it to the $j$th position, test $\gamma_j$ against the threshold and accept if it passes the test, as illustrated in (\ref{eq:dimer2}). Another option is to discard all previous pivoting and start a completely new one by determining the largest column for the first position, and proceed with a completely new pivoting process on the updated contents of $\Lw$. 
The value of $k$ is not necessarily fixed and may change dynamically with a safety device to prevent failure.
With proper data structure, one can develop a detailed algorithm and an efficient software implementation. For the sake of brevity, we omit the details. 
However, we provide one illustrative example. 
\begin{example}\label{EX:dimer}
	{\em
		Let $\f(t;\mu) = 10 e^{-\mu t}(\cos(4\mu t) + \sin(4\mu t))$, $1\leq t\leq 6$, $0\leq\mu\leq \pi$. Take $40$ uniformly sampled values of $\mu$ and compute the snapshots over the discretized $t$--domain at $n=10000$ uniformly spaced nodes.
		The best low rank approximation of the sampled $10000\times 40$	returned $\U$
		with $m=34$ columns. This indicates that the POD basis has captured  the function's behavior. 
		We allowed \dimer to process only $k=m$ columns in the work array $\Lw$, and set the upper bound for $\C$ at\footnote{This \emph{ad hoc} choice is motivated by the structure of the upper bound for $\C$, as in the proof of Theorem \ref{Theorem:2mbound}. Note that we use $\sqrt{m}$ instead of the worst case theoretical $O(2^m)$ bound.} $\sqrt{m}\sqrt{n-m+1}$. Column index drawing is done simply: $\ell$ ``random" indices are taken as the $\ell$ leading indices of randomly permuted active set. We stress here that the purpose of this example is to illustrate the idea and its potential, to
		motivate further study of the randomized sampling approach to \deim projection.
		
		After processing $113$ rows of $\U$ (out of $10000$), \dimer selected a submatrix with $\C\approx 181.45$; \deim processed the whole matrix $\U$
		and returned $\C\approx 79.13$. To test how well the two methods approximate $\f$, we compute its value at $200$ points in the $\mu$--interval: for each $\mu_j$ the function is evaluated over the $t$--grid  giving $f_{\mu_j}\in\mathbb{R}^n$. The same is done with \deim and \dime projections giving $f_{\mu_j}^{\deim}$ and $f_{\mu_j}^{\dime}$, respectively. 
		The results of a comparison are depicted in Figure \ref{FIG1dimer}.
		\begin{figure}[hhh]
			\begin{center}
		\includegraphics[width=2.3in,height=2.0in]{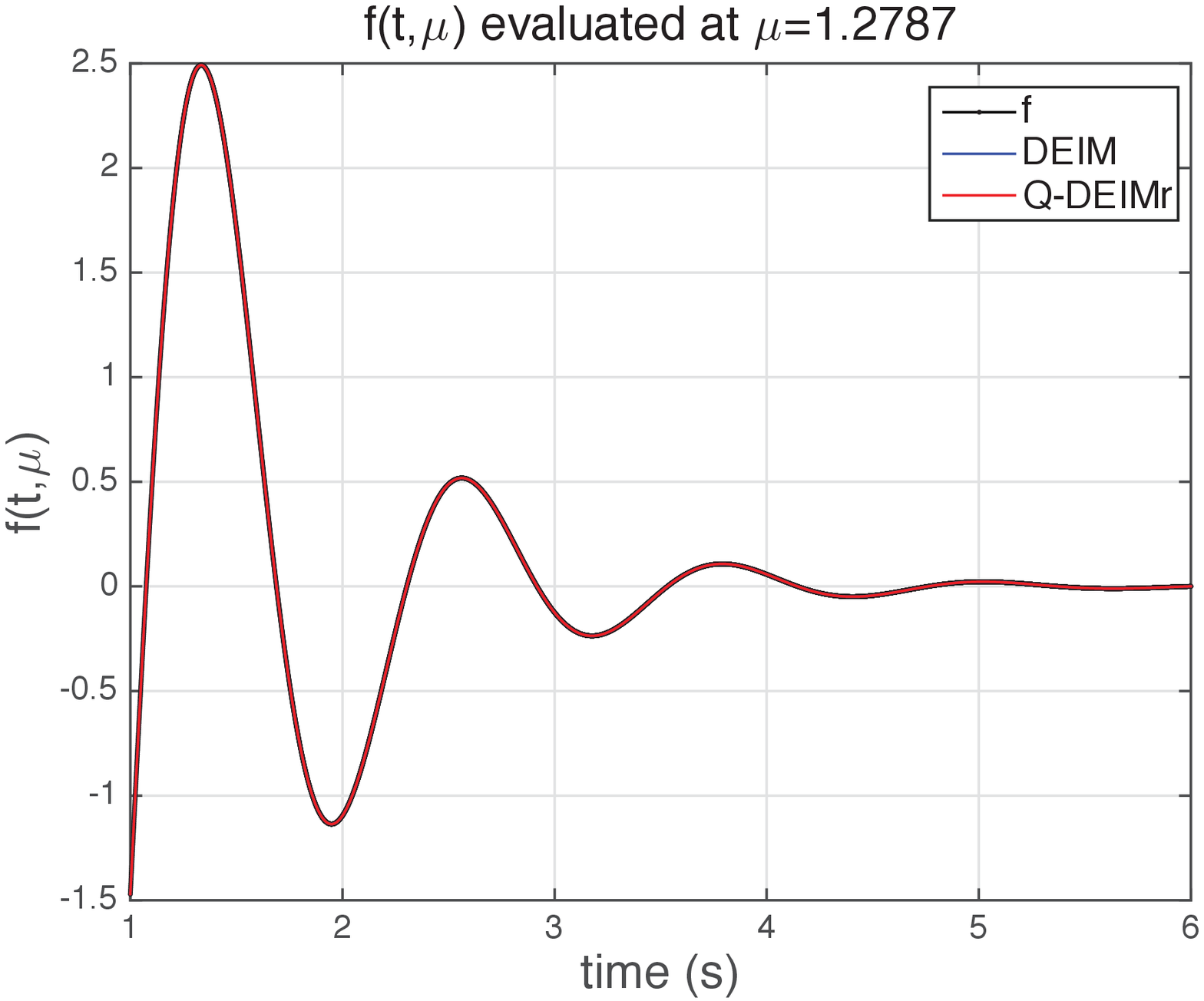}
		\hspace{1ex}
		\includegraphics[width=2.3in,height=2.0in]{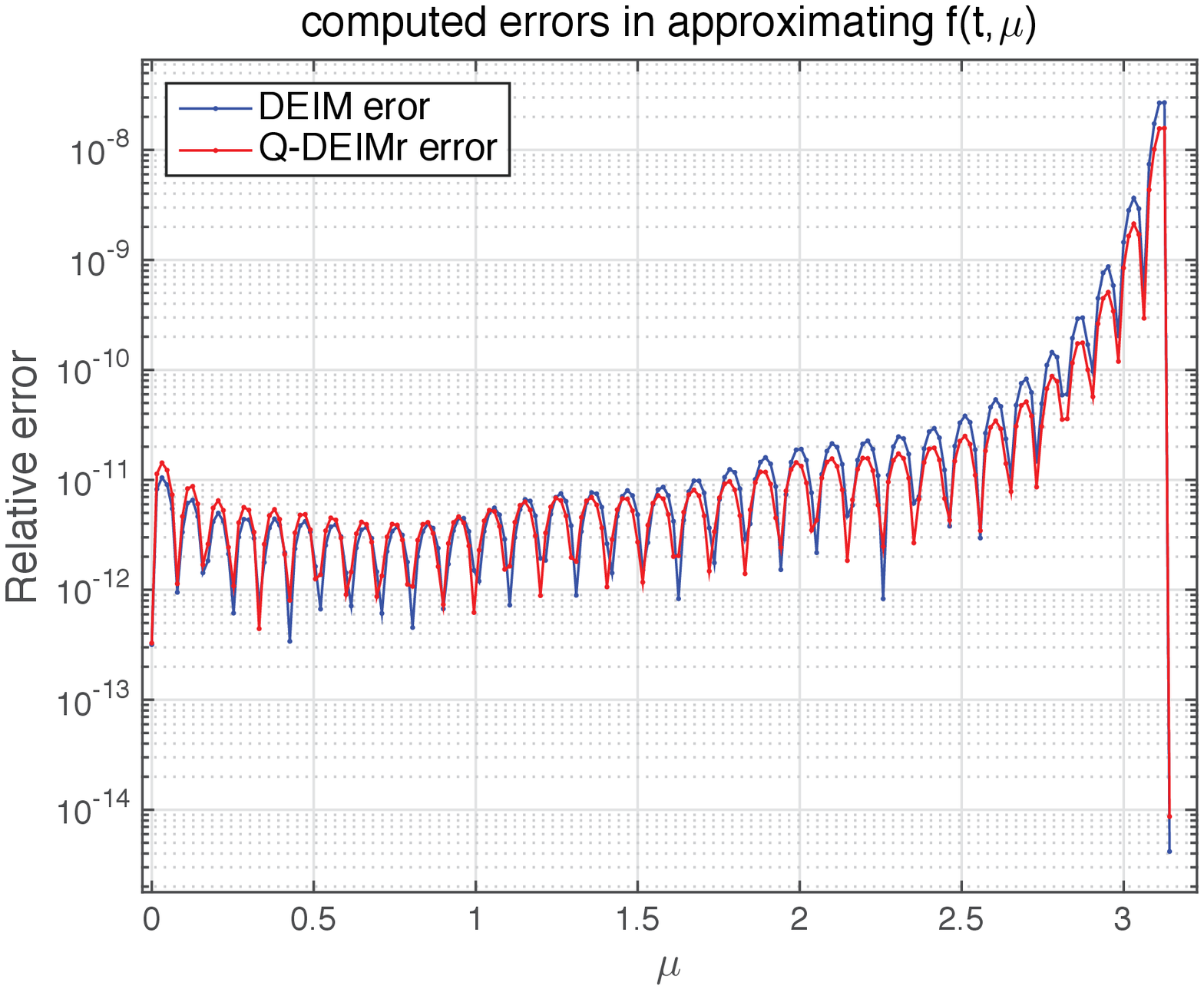}
			\end{center}
			\caption{\label{FIG1dimer} (Example \ref{EX:dimer}) Comparison of the approximation errors of \deim and \dimer for $\f(t;\mu) = 10 e^{-\mu t}(\cos(4\mu t) + \sin(4\mu t))$. Left figure: The function evaluated at $\mu=1.2787$.
				Right figure: The relative errors $\|f_{\mu_j}^{\deim} - f_{\mu_j}\|_2/\|f_{\mu_j}\|_2$, $\|f_{\mu_j}^{\dime} - f_{\mu_j}\|_2/\|f_{\mu_j}\|_2$ for $200$ uniformly spaced values of $\mu_j\in[0,\pi]$. \dimer used $113$ rows (at most $34$ at the same time) of $\U$ to make a selection; \deim used all $10000$ rows.}
		\end{figure} 
		\begin{figure}[hhh]
			\begin{center}
	\includegraphics[width=2.3in,height=2.0in]{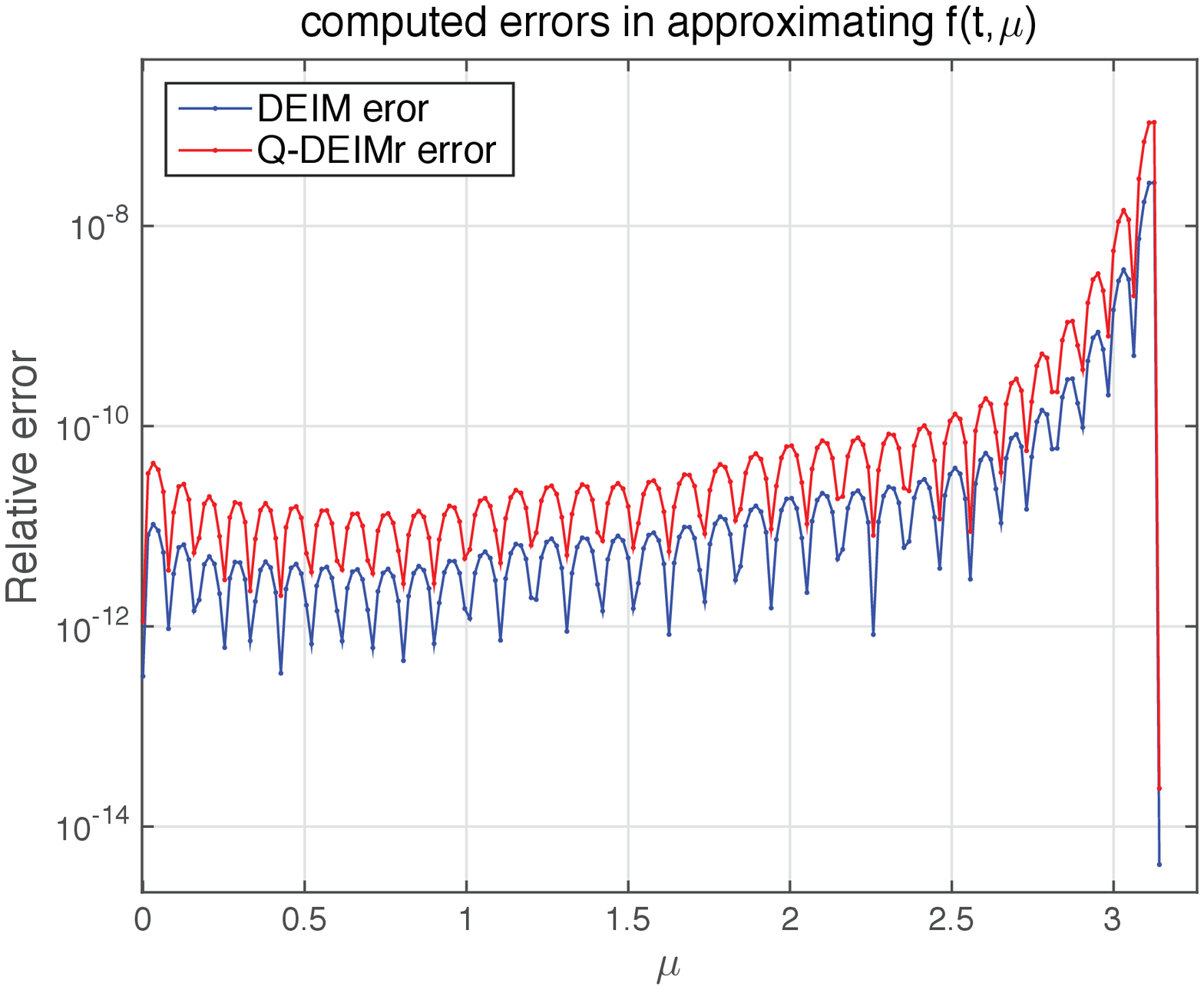}
	\hspace{1ex}
	\includegraphics[width=2.3in,height=2.0in]{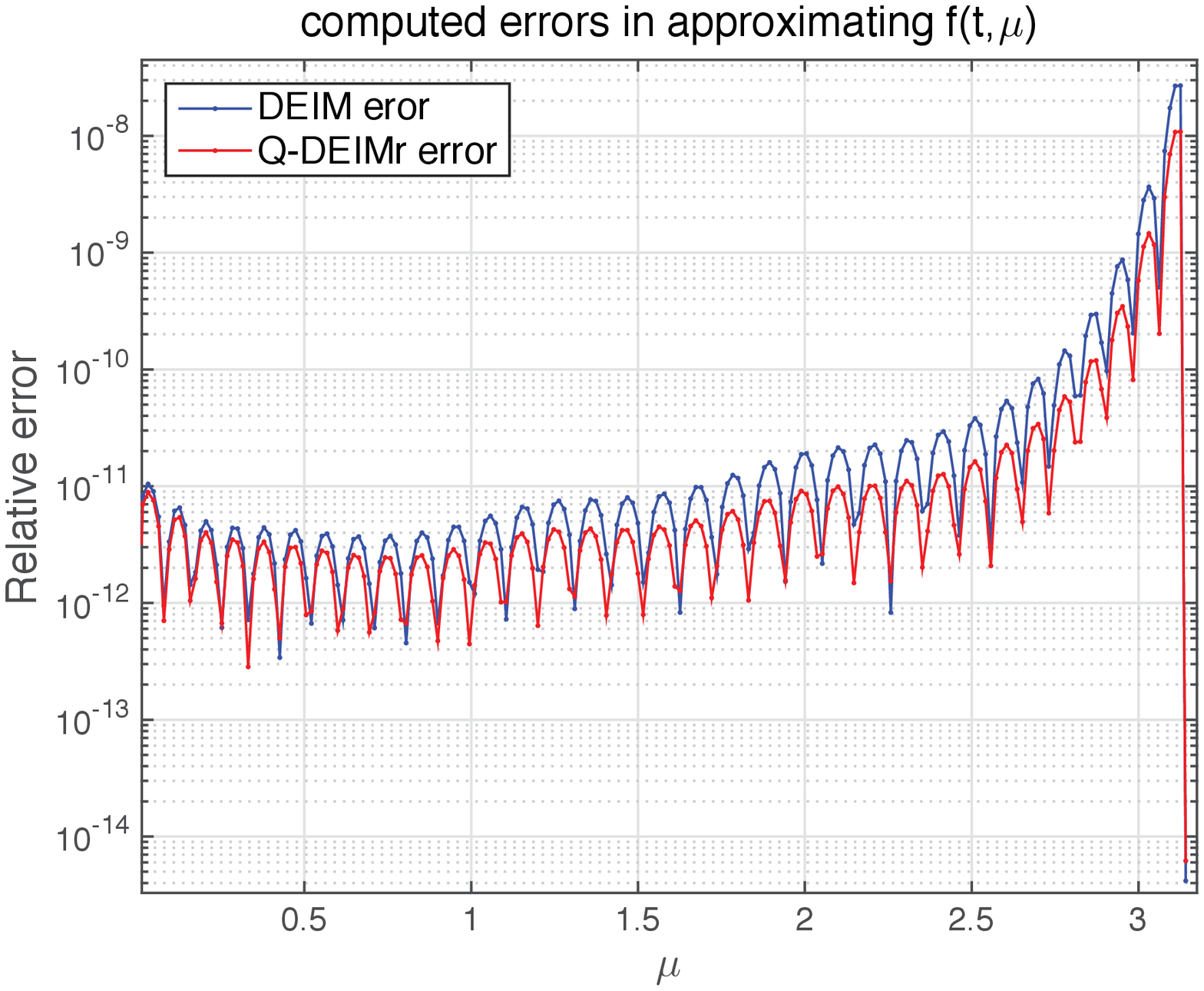}
			\end{center}
			\caption{\label{FIG2dimer} (Example \ref{EX:dimer}) Comparison of the relative errors $\|f_{\mu_j}^{\deim} - f_{\mu_j}\|_2/\|f_{\mu_j}\|_2$, $\|f_{\mu_j}^{\dime} - f_{\mu_j}\|_2/\|f_{\mu_j}\|_2$ for $\f(t;\mu) = 10 e^{-\mu t}(\cos(4\mu t) + \sin(4\mu t))$. Left figure: Upper bound in \dimer set to $m\sqrt{n-m+1}$; it used $53$ rows with $\C\approx 2532.9$.
				Right figure: Upper bound in \dimer set to $\sqrt{m}\sqrt{n-m+1}/5$; it used $220$ rows with $\C\approx 103.1$. In both cases, \deim used all $10000$ rows of $\U$ to make a selection, and \dimer was allowed to process at most $34$ rows of $\U$ at the same time.}
		\end{figure} 	
Since at this point no sophisticated sampling strategy is used, the results may vary, depending on $n$, $m$, and the given upper bound for $\C$. Figure \ref{FIG2dimer} illustrates how the prescribed upper bound for $\C$ changes the execution and performance of \dimer. By visiting only  $220$ rows, we fully recover the accuracy 
achieved by using  all  $10000$ rows of $\U$.
	}	
\end{example}

\begin{remark}{\em 
The technical details of using ICE in the above procedure are similar to the rank revealing method with windowed column pivoting \cite{bischof-q-orti-RRQR-1998}. However, the overall procedure is substantially different in spirit. 
The difference is that our objective is not to compute a rank revealing QR factorization (we know that $\W$ is of full row rank, $\W\W^*=\Id_m$, and we do not even need its QR factorization), but just to find a well conditioned submatrix. This allows to touch only a selection of columns of $\W$ and exit when a sufficiently well conditioned submatrix is determined. 
}
\end{remark}
\begin{remark}{\em 
 The quest for a well-conditioned submatrix of $\W$ can be obviously parallelized, and many processors can independently work on different  (not necessarily disjoint)  subsets of column indices, with no need whatsoever to engage in communication, until one finds suitable columns and sends a halt signal. If more than one selection is found, the best one will be chosen. 
	}
\end{remark}

The column selection procedure can be improved at a cost of one pass through the array $\W$ to compute the column norms $\omega_i = \| \W(:,i) \|_2$. Such  additional information can be used e.g. in the following two ways:

$\bullet$ Select from the sorted columns in batches, as needed, starting from the largest ones. The norms of the columns in the active set can also be down-dated, using the procedure from column pivoted QR with numerically robust implementation  \cite{drmac-bujanovic-2008}.

$\bullet$ Define $p_i = \omega_i^2/(\sum_{j=1}^n\omega_j^2)$, $i=1,\ldots, n$. Then
$(p_1,\ldots,p_n)$ is a probability distribution that can be used to draw column samples. It prefers larger columns. Each column is used only once, and the distribution is computed for the active set. 

This opens a completely new aspect of \deim and establishes its connection to randomized numerical linear algebra, in particular with randomized sampling of rows of orthonormal matrices \cite{Ipsen:2014:ECS}. 
In particular, one can view \dimer as a guided randomized sampling algorithm for orthonormal matrices. 
Detailed analysis of blending the two procedures is omitted here; it will be available in our subsequent work.   

\begin{remark}\label{REM:extra}{\em 
Our experiments with randomized selection indicate that merely picking $m$ random interpolation indices will not work in general and sophisticated strategies of \deim, \dime, \dimer are necessary for a reliable and robust black-box procedure. Indeed, this is already revealed in Example \ref{EX:dimer}; the initial random selection is not enough and \dimer brings in additional rows to process. Clearly, the results depend on the function being approximated but just to illustrate the importance of the selection principle of \deim and \dimer, below we show the reconstruction accuracy of \deim and randomly selected indices (with few trials to select better rows, but without the condition number control as in \dimer) for two
nonlinear parametrized functions. \textcolor{black}{As the figures illustrate, the random selection without the techniques from the DEIM procedures may perform very poorly.}	
\begin{figure}[hhh]
	\begin{center}
		\includegraphics[width=2.3in,height=2.0in]{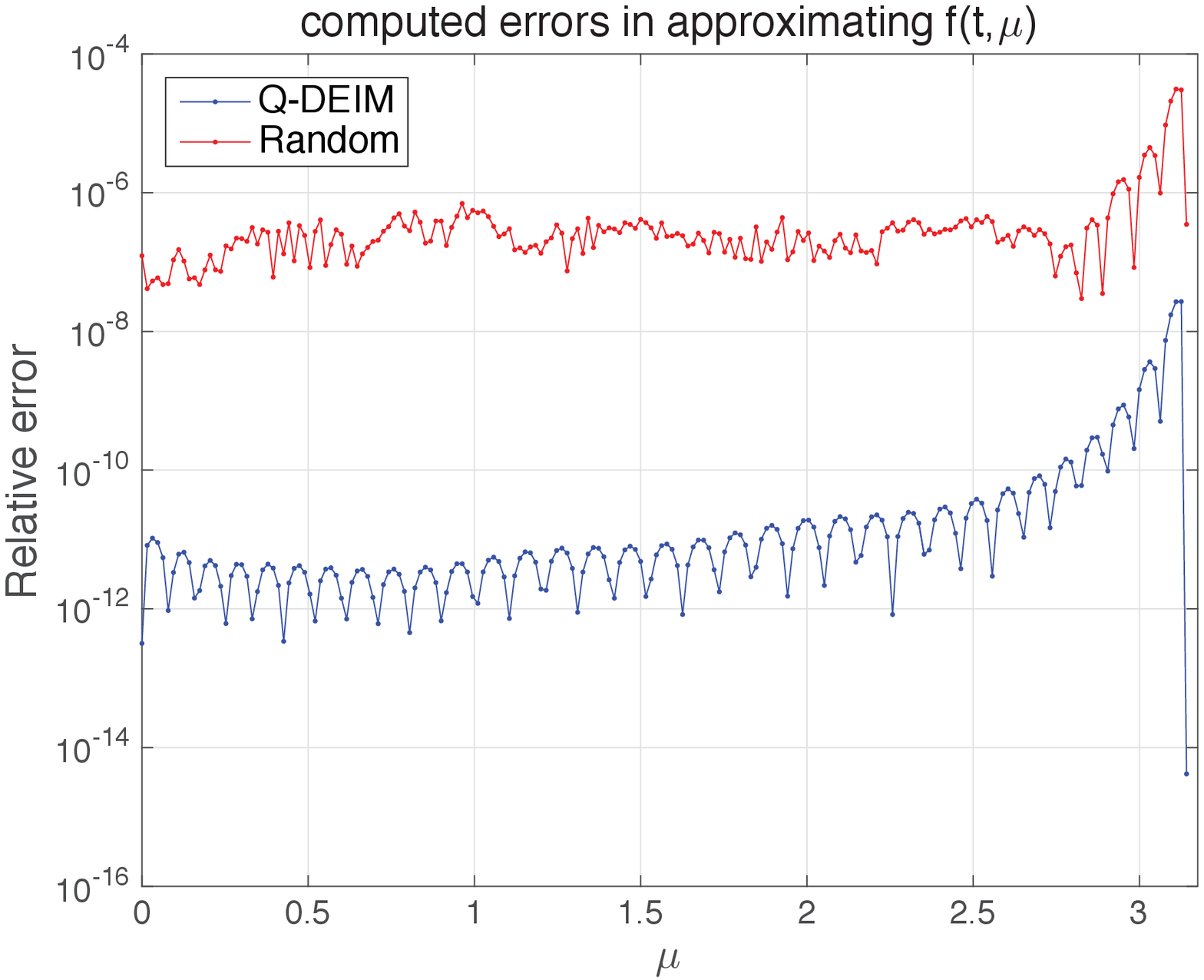}
		\includegraphics[width=2.3in,height=2.0in]{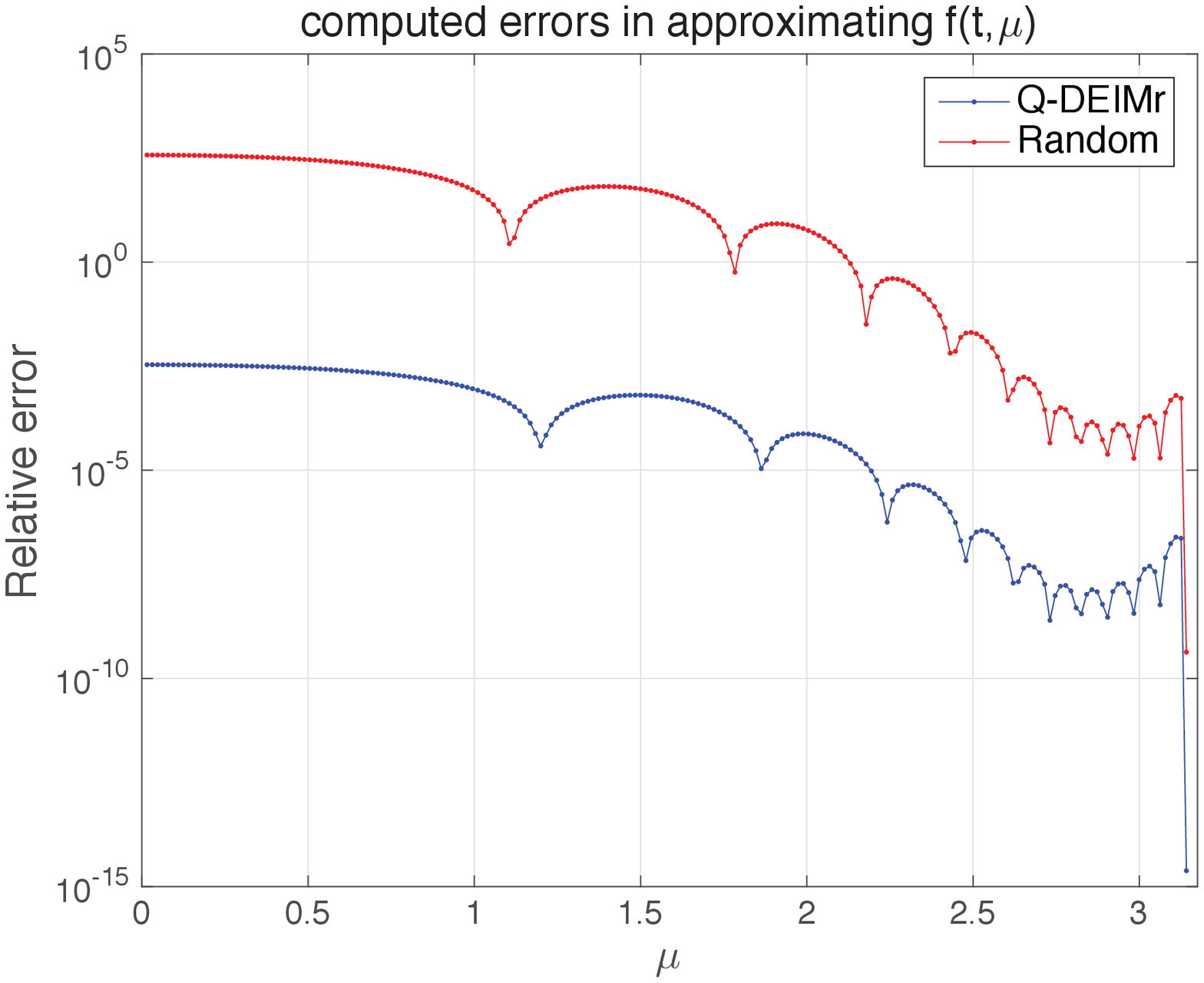}
	\end{center}
	
	\caption{(Remark \ref{REM:extra}) The reconstruction accuracy of a plain random selection of indices, as compared to \dime and \dimer. The first plot uses the data generated by the 
		function from Example \ref{EX:dimer}. The function used for te second plot is $f(x,\mu)=\sinh((\mu*\cosh(\mu/x)))$, $0.1\leq x\leq 6$, $0\leq \mu\leq \pi$, and $m=11$ out of $n=2000$ indices are selected.}
\end{figure} 
		}
\end{remark}

\section{Conclusions and Future Work}\label{S=Conclusion}
Using the tools from QR factorization with column pivoting, this paper has introduced a new \deim index selection strategy, that is  invariant under orthogonal transformations, with a sharper error bound for the \deim projection error. The new approach, called \dime, is tested on several numerical examples and it performs as well as the original \deim selection procedure. For the cases of large dimensions, 
a modification is proposed that uses only randomly sampled rows of the data matrix yet still leading to high-fidelity  approximations.

In addition to the nonlinear model reduction and parametrized function settings  we presented here,  the new \dime selection is well suited for several important applications that we are currently investigating for our  subsequent work. One such application is randomized sampling of rows of orthonormal matrices for solving the least-squares problems effectively in the cases with huge row dimension. The second one is the \deim induced CUR factorization recently introduced by Sorensen and Embree \cite{DEIM-CUR}. A third application arises in 
nonlinear inversion and parametric model reduction, see, e.g., \cite{BGWSIREV,DOT2015}, where an affine decomposition is needed for a parametric 
matrix  $A(p) \in \mathbb{R}^{n \times n}$  for efficient online model reduction step. This is usually handled by vectorizing $A(p)$ which might lead to very large row dimension depending on the sparsity pattern of $A(p)$. We are currently testing \dime and \dimer in these settings. 
\textcolor{black}{Further, both \deim and \dime (including \dimer) can adopt an updating scheme  when the nonlinear snapshot basis is obtained by the SVD: one can enlarge $m$ and $\U$ incrementally to yield better approximations. In \dime, increasing $m$ will require updating a rank-revealing QR-decomposition; thus finding an efficient updating scheme for this incremental implementation will prove very useful.}

\section{Acknowledgments}
We would like to thank Dr.\ Saifon  Chaturantabut for providing the data and code for the FitzHugh-Naguma Model Example in \S \ref{ex:fn}, \textcolor{black}{and Dr.\ Christopher Beattie, Dr.\ Mark Embree and Dr.\ Danny Sorensen} for various enlightening discussions. 
The work of Drma\v{c}  was supported by the grant
		HRZZ-9345 from the Croatian Science Foundation, in
		part by the NSF through grant DMS-1217156, and by the
		Interdisciplinary Center for Applied Mathematics
		during his research visit to the Department of Mathematics at Virginia Tech in the 2014-2015 Academic Year. The work of Gugercin was supported in
		part by the NSF through grant DMS-1217156.
The authors acknowledge partial support of the Einstein Foundation and of the Technische Universit\"{a}t Berlin during the completion of this work, and we thank Dr.\ Christopher Beattie and Dr.\ Volker Mehrmann for hosting us in Berlin. 		

\bibliography{bibDEIM}

\begin{thebibliography}{10}

\bibitem{LAPACK}
{\sc E.~Anderson, Z.~Bai, C.~Bischof, L.~S. Blackford, J.~Demmel, J.~J.
  Dongarra, J.~Du~Croz, S.~Hammarling, A.~Greenbaum, A.~McKenney, and
  D.~Sorensen}, {\em LAPACK Users' Guide (Third Ed.)}, Society for Industrial
  and Applied Mathematics, Philadelphia, PA, USA, 1999.

\bibitem{antil2013two}
{\sc H.~Antil, S.~E. Field, F.~Herrmann, R.~H. Nochetto, and M.~Tiglio}, {\em
  Two-step greedy algorithm for reduced order quadratures}, Journal of
  Scientific Computing, 57 (2013), pp.~604--637.

\bibitem{Astrid2008}
{\sc P.~Astrid, S.~Weiland, K.~Willcox, and T.~Backx}, {\em Missing point
  estimation in models described by proper orthogonal decomposition}, IEEE
  Transactions on Automatic Control,  (2008), pp.~2237--2251.

\bibitem{bai2006projection}
{\sc Z.~Bai and D.~Skoogh}, {\em A projection method for model reduction of
  bilinear dynamical systems}, Linear Algebra and its Applications, 415 (2006),
  pp.~406--425.

\bibitem{barrault04-EIM}
{\sc M.~Barrault, N.~C. Nguyen, Y.~Maday, and A.~T. Patera}, {\em An empirical
  interpolation method: Application to efficient reduced-basis discretization
  of partial differential equations}, C. R. Acad. Sci. Paris, S\'{e}rie I., 339
  (2004), pp.~667--672.

\bibitem{BreitenBenner2012b}
{\sc P.~Benner and T.~Breiten}, {\em Interpolation-based $\mathcal{H}_2$-model
  reduction of bilinear control systems}, SIAM Journal on Matrix Analysis and
  Applications, 33 (2012), pp.~859--885.

\bibitem{benner2015two}
{\sc P.~Benner and T.~Breiten}, {\em Two-sided projection methods for nonlinear
  model order reduction}, SIAM Journal on Scientific Computing, 37 (2015),
  pp.~B239--B260.

\bibitem{BGWSIREV}
{\sc P.~Benner, S.~Gugercin, and K.~Willcox}, {\em A survey of model reduction
  methods for parametric systems}, Tech. Rep. MPIMD/13-14, Max Planck Institute
  Magdeburg Preprint, August 2013.

\bibitem{berkooz}
{\sc G.~Berkooz, P.~Holmes, and J.~Lumley}, {\em {The proper orthogonal
  decomposition in the analysis of turbulent flows}}, Annual Review of Fluid
  Mechanics, 25 (1993), pp.~539--575.

\bibitem{bischof-q-orti-RRQR-1998-TOMS782}
{\sc C.~H. Bischof and G.~Quintana-Orti}, {\em Algorithm 782: codes for
  rank--revealing {QR} factorizations of dense matrices}, {ACM Transactions on
  Mathematical Software}, 24 (1998), pp.~254--257.

\bibitem{bischof-q-orti-RRQR-1998}
\leavevmode\vrule height 2pt depth -1.6pt width 23pt, {\em Computing
  rank--revealing {QR} factorizations of dense matrices}, {ACM Transactions on
  Mathematical Software}, 24 (1998), pp.~226--253.

\bibitem{ScaLAPACK}
{\sc L.~S. Blackford, J.~Choi, A.~Cleary, E.~D'Azeuedo, J.~Demmel, I.~Dhillon,
  S.~Hammarling, G.~Henry, A.~Petitet, K.~Stanley, D.~Walker, and R.~C.
  Whaley}, {\em ScaLAPACK User's Guide}, Society for Industrial and Applied
  Mathematics, Philadelphia, PA, USA, 1997.

\bibitem{bus-gol-65}
{\sc P.~A. Businger and G.~H. Golub}, {\em Linear least squares solutions by
  {Householder} transformations}, Numerische Mathematik, 7 (1965),
  pp.~269--276.

\bibitem{Carlberg2013}
{\sc K.~Carlberg, C.~Farhat, J.~Cortial, and D.~Amsallem}, {\em The {GNAT}
  method for nonlinear model reduction: Effective implementation and
  application to computational fluid dynamics and turbulent flows}, Journal of
  Computational Physics, 242 (2013), pp.~623--647.

\bibitem{chandras-ipsen-rrqr-94}
{\sc S.~Chandrasekaran and I.~C.~F. Ipsen}, {\em On rank--revealing
  factorizations}, SIAM Journal on Matrix Analysis and Applications, 15 (1994),
  pp.~592--622.

\bibitem{Chen1999}
{\sc Y.~Chen}, {\em Model order reduction for nonlinear systems}, Master's
  thesis, Massachusetts Institute of Technology, 1999.

\bibitem{condon2004model}
{\sc M.~Condon and R.~Ivanov}, {\em Model reduction of nonlinear systems},
  COMPEL-The international journal for computation and mathematics in
  electrical and electronic engineering, 23 (2004), pp.~547--557.

\bibitem{DOT2015}
{\sc E.~de~Sturler, S.~Gugercin, M.~Kilmer, S.~Chaturantabut, C.~Beattie, and
  M.~O'Connell}, {\em Nonlinear parametric inversion using interpolatory model
  reduction}, SIAM Journal on Scientific Computing, to appear; arXiv:1311.0922,
   (2015).

\bibitem{drmac-block-jacobi}
{\sc Z.~Drma\v{c}}, {\em A global convergence proof for cyclic {Jacobi} methods
  with block rotations}, SIAM Journal on Matrix Analysis and Applications, 31
  (2009), pp.~1329--1350.

\bibitem{drmac-bujanovic-2008}
{\sc Z.~Drma\v{c} and Z.~Bujanovi\'{c}}, {\em On the failure of rank revealing
  {QR} factorization software -- a case study}, ACM Trans. Math. Softw., 35
  (2008), pp.~1--28.

\bibitem{Everson1995}
{\sc R.~Everson and L.~Sirovich}, {\em {The {Karhunen-Loeve} Procedure for
  Gappy Data}}, Journal of the Optical Society of America, 12 (1995),
  pp.~1657--1664.

\bibitem{FKF-1968}
{\sc D.~{Faddeev}, V.~{Kublanovskaya}, and V.~{Faddeeva}}, {\em {Solution of
  linear algebraic systems with rectangular matrices.}}, {Proc. Steklov Inst.
  Math.}, 96 (1968), pp.~93--111.

\bibitem{flagg2013multipoint}
{\sc G.~Flagg and S.~Gugercin}, {\em Multipoint {V}olterra series interpolation
  and {$\mathcal{H}_2$} optimal model reduction of bilinear systems}, To appear
  in SIAM Journal on Matrix and Analysis and Applications,  (2015).
\newblock Available as arXiv preprint arXiv:1312.2627.

\bibitem{glover1984all}
{\sc K.~Glover}, {\em All optimal {Hankel}-norm approximations of linear
  multivariable systems and their $l_\infty$-error bounds}, International
  journal of control, 39 (1984), pp.~1115--1193.

\bibitem{Goreinov19971}
{\sc S.~Goreinov, E.~Tyrtyshnikov, and N.~Zamarashkin}, {\em A theory of
  pseudoskeleton approximations}, Linear Algebra and its Applications, 261
  (1997), pp.~1 -- 21.

\bibitem{gu2011qlmor}
{\sc C.~Gu}, {\em {QLMOR}: a projection-based nonlinear model order reduction
  approach using quadratic-linear representation of nonlinear systems},
  Computer-Aided Design of Integrated Circuits and Systems, IEEE Transactions
  on, 30 (2011), pp.~1307--1320.

\bibitem{gugercin2008hmr}
{\sc S.~Gugercin, A.~Antoulas, and C.~Beattie}, {\em $\mathcal{H}_2$ model
  reduction for large-scale linear dynamical systems}, SIAM Journal on Matrix
  Analysis and Applications, 30 (2008), pp.~609--638.

\bibitem{hinze2005proper}
{\sc M.~Hinze and S.~Volkwein}, {\em Proper orthogonal decomposition surrogate
  models for nonlinear dynamical systems: Error estimates and suboptimal
  control}, in Dimension Reduction of Large-Scale Systems, Springer, 2005,
  pp.~261--306.

\bibitem{hotelling}
{\sc H.~Hotelling}, {\em {Analysis of a complex of statistical variables with
  principal components}}, Journal of Educational Psychology, 24 (1933),
  pp.~417--441,498--520.

\bibitem{Ipsen:2014:ECS}
{\sc I.~C.~F. Ipsen and T.~Wentworth}, {\em The effect of coherence on sampling
  from matrices with orthonormal columns, and preconditioned least squares
  problems}, SIAM Journal on Matrix Analysis and Applications, 35 (2014),
  pp.~1490--1520.

\bibitem{kahan-66}
{\sc W.~Kahan}, {\em Numerical linear algebra}, Canadian Mathematical Bulletin,
  9 (1965), pp.~757--801.

\bibitem{kunisch2002galerkin}
{\sc K.~Kunisch and S.~Volkwein}, {\em {Galerkin} proper orthogonal
  decomposition methods for a general equation in fluid dynamics}, SIAM Journal
  on Numerical analysis, 40 (2002), pp.~492--515.

\bibitem{law-han-74}
{\sc C.~L. Lawson and R.~J. Hanson}, {\em Solving Least Squares Problems},
  Prentice--Hall Inc., Englewood Cliffs, N. J., 1974.

\bibitem{loeve}
{\sc M.~Lo\'{e}ve}, {\em {Probability Theory}}, D. Van Nostrand Company Inc.,
  New York, 1955.

\bibitem{lumley}
{\sc J.~Lumley}, {\em {The Structures of Inhomogeneous Turbulent Flow}},
  Atmospheric Turbulence and Radio Wave Propagation,  (1967), pp.~166--178.

\bibitem{moore1981principal}
{\sc B.~Moore}, {\em Principal component analysis in linear systems:
  Controllability, observability, and model reduction}, IEEE Transactions on
  Automatic Control, 26 (1981), pp.~17--32.

\bibitem{mullis1976synthesis}
{\sc C.~Mullis and R.~Roberts}, {\em Synthesis of minimum roundoff noise fixed
  point digital filters}, IEEE Transactions on Circuits and Systems, 23 (1976),
  pp.~551--562.

\bibitem{dansiam10}
{\sc S.~S.~Chaturantabut and D.~Sorensen.}, {\em Nonlinear model reduction for
  porous media flow}, in 2010 SIAM Annual Meeting, 2010.

\bibitem{DEIM}
{\sc S.~S.~Chaturantabut and D.~Sorensen}, {\em Nonlinear model reduction via
  discrete empirical interpolation}, SIAM Journal on Scientific Computing, 32
  (2010), pp.~2737--2764.

\bibitem{dan}
{\sc D.~Sorensen}.
\newblock Private communications, 2010.

\bibitem{DEIM-CUR}
{\sc D.~C. Sorensen and M.~Embree}, {\em {A {DEIM} induced {CUR}
  factorization}}, {CAAM Department Technical Report TR14-04}, Rice University,
  July 2014.

\end{thebibliography}
\bibliographystyle{siam}
\end{document}